\documentclass{article}

\usepackage[english]{babel}

\usepackage[utf8]{inputenc}

\usepackage[scale=.7]{geometry}
         
\usepackage{graphicx}
\usepackage{amsmath,amssymb,amsthm}
\usepackage{hyperref}
\usepackage{enumitem}  
\usepackage[nottoc]{tocbibind}
\usepackage{ifthen} 
\usepackage{iflang}
\usepackage{blindtext}
\usepackage{float}
\usepackage{xcolor}
\usepackage{comment}
\usepackage{bbm}
\usepackage{tikz-cd}
\usepackage{caption}
\usepackage{subcaption}

\def\thesistitle{Spectral truncations of $\SU(2)$ and the spectral localizer}
\def\thesisauthorfirst{Matt}
\def\thesisauthorsecond{Verhoeven}
\def\thesisdate{\today}

\title{\thesistitle}
\author{\thesisauthorfirst\space\thesisauthorsecond}
\date{\thesisdate}

\numberwithin{equation}{section} 
\newtheorem{theorem}[subsection]{Theorem}
\newtheorem{corollary}[subsection]{Corollary}
\newtheorem{lemma}[subsection]{Lemma}
\newtheorem{proposition}[subsection]{Proposition}
\newtheorem{conjecture}[subsection]{Conjecture}
\theoremstyle{definition}
\newtheorem{definition}[subsection]{Definition}
\newtheorem{remark}[subsection]{Remark}

\newcommand{\SU}[0]{\mathrm{SU}}

\newcommand{\C}{\mathbb{C}}
\newcommand{\R}{\mathbb{R}}
\newcommand{\N}{\mathbb{N}}
\newcommand{\Z}{\mathbb{Z}}
\newcommand{\prop}[0]{\mathrm{prop}}

\newcommand{\trunc}[2][]{
	\ifthenelse{\equal{#1}{}}{[{#2}]}{[{#2}]_{#1}}
}
\newcommand{\spintrunc}[2][]{
	\ifthenelse{\equal{#1}{}}{[\![{#2}]\!]}{[\![{#2}]\!]_{#1}}
}
\newcommand{\tensor}{\otimes}
\newcommand{\II}{\mathcal{I}}
\DeclareMathOperator{\End}{End}

\newcommand{\J}{\II(J)}
\DeclareMathOperator{\spann}{span}
\DeclareMathOperator{\sign}{sign}
\newcommand{\twoJ}{\II(2J)}
\DeclareMathOperator{\Hom}{Hom}
\newcommand{\jrel}[2]{|{#1}| \preceq {#2}}
\newcommand{\jrelplus}[2]{|{#1}|_{+} \preceq {#2}}
\newcommand{\jrelminus}[2]{|{#1}|_{-} \preceq {#2}}

\DeclareMathOperator{\ind}{ind}
\DeclareMathOperator{\Sig}{Sig}

\DeclareMathOperator{\Tr}{Tr}

\DeclareMathOperator{\Spin}{Spin}

\begin{document}
\maketitle

\begin{abstract}
We study operator systems arising from Peter-Weyl and Dirac spectral truncations of $\SU(2)$ in the sense of Connes-Van Suijlekom. We determine operator system invariants for these systems, including the $C^*$-envelope and uniform bounds on the propagation number. For Peter-Weyl truncations, we find $\prop \leq 2$ for all connected compact Lie groups, generalizing known results for the torus. We also carry out an analysis of the spectral localizer by Loring and Schulz-Baldes on $\SU(2)$. In this analysis, we find numerical evidence that an index map for operator systems previously proposed by Van Suijlekom only captures the right information in a suitable interval of the parameter $\kappa$, instead of the limit $\kappa \downarrow 0$. 
\end{abstract}

\section{Introduction}
Within noncommutative geometry, Connes and Van Suijlekom introduced the framework of spectral truncations as a method to approximate noncommutative spaces using a finite amount of spectral information \cite{ConnesvS2020}, which builds on earlier work by \cite{Dandrea2014}. The role played by $C^*$-algebras and spectral triples $(\mathcal A, \mathcal H, D)$ is taken over by operator systems and truncated operator system spectral triples $(P\mathcal A P, P \mathcal H, PDP)$ where $P$ is some spectral projection of $D$. Here $P\mathcal A P$ does not come with a multiplication operator anymore, so it is not a $C^*$-algebra. Nevertheless, it is a linear subspace of $B(P\mathcal H)$ closed under taking adjoints, and hence it is an operator system. In the same paper, the propagation number was introduced as an invariant for operator systems, which is used for measuring how far an operator system is removed from being a $C^*$-algebra.\\
\\
Building on this work, several advances have been made in recent years. The convergence of state spaces of (Dirac) spectral truncations to their original spectral triples has been determined for the circle $S^1$ in \cite{vanSuijlekom2021, Hekkelman2021}, for more general tori $\mathbb{T}^d$ in \cite{Leimbach24} and in a recent preprint for noncommutative tori in \cite{Bhatt2026}. Furthermore, another preprint proves convergence of spectral truncations for compact quantum groups \cite{Peng2026}. Similar results have been obtained for other types of truncations. In \cite{GaudillotEstrada2025,Leimbach2025} it was proven that for compact quantum groups the state spaces of Peter-Weyl truncated operator systems converge to the state space of the quantum group in an appropriate sense. For Fourier truncations, it was proven in \cite{Rieffel2023} that the state spaces converge for compact quantum groups and finitely generated groups. \\
\\
On the other hand, work has been done on the operator system structure of these truncations. The propagation number has been determined for spectral truncations of the circle $S^1$ and its dual operator system in \cite{ConnesvS2020,Farenick2021}, spectral truncations of the torus $\mathbb{T}^d$ in \cite{Leimbach24}, truncations of the finite cyclic group and the operator system of circulant matrices in \cite{Koot2021}, operator systems associated to tolerance relations \cite{Connes2022} and for operator systems associated to graphs in \cite{Gielen2023}. In particular, the propagation number of the truncated torus is $2$ irrespective of the dimension $d$, and this is the only manifold so far for which the propagation number of its spectral truncations has been determined. This result is extended to noncommutative tori in a recent preprint \cite{Bhatt2026}.\\
\\
In parallel to these studies, the spectral localizer has been developed by Loring and Schulz-Baldes in a series of papers \cite{Loring2017, Loring2019, Loring2020} and by others (see Chapter 10 the book \cite{SchulzBaldes2023} for an overview). This method allows for the computation of $K$-theoretic invariants via the index pairing by using only a finite amount of spectral information.
For a given tuning parameter $\kappa$ and truncation parameter $\rho$, one can define the spectral localizer matrix $L_{\kappa, \rho}$. Loring and Schulz-Baldes proved that the index pairing is equal to $\frac{1}{2}\Sig(L_{\kappa, \rho})$ for $\kappa$ small enough and $\rho$ large enough.
Although originally motivated by use in condensed matter theory \cite{Prodan2016}, the spectral localizer has been applied in the context of spectral truncations \cite{Hekkelman2021,VanSuijlekom2024}, and alternative proofs of correctness of the spectral localizer have been found based on $KK$-theory and $E$-theory \cite{vdDungen2026, Li2026}. In applications, it has been observed that the actual range of $\kappa$ and $\rho$ which gives the index is more lenient than the range proven in \cite{SchulzBaldes2023}, and this range is provably improved for the even localizer in a preprint \cite{CerjanSB2026}. So far, operator system spectral triples for which this behaviour of the spectral localizer has been determined are the circle $S^1$ \cite{Hekkelman2021} and the 2-torus $\mathbb{T}^2$ \cite{VanSuijlekom2024b}. \\
\\
In this paper we add to this set of examples by considering the Lie group $\SU(2)$ and the associated spectral triple $(C^{\infty}(\SU(2)), L^2(S), D)$. Unlike the circle and torus, this is a noncommutative Lie group, which complicates calculations. For both the Peter-Weyl and the spectral truncation we determine the dimensions of truncated operator systems as a function of the truncation parameter. Furthermore, we determine the $C^*$-envelopes and provide uniform bounds on the propagation numbers of these systems. For finite-dimensional Peter-Weyl truncations $R_{\Lambda}$ of connected compact Lie groups $G$ we find that $\prop(R_{\Lambda}) \leq 2$, generalizing the result for tori from \cite{Leimbach24}. The proof also implies more generally that the propagation number of a truncation based on the Laplace operator cannot exceed 2 for nice enough manifolds, such as spheres $S^n$. Furthermore, we find a propagation number of $4$ for (Dirac) spectral truncations with a small truncation parameter, which makes $\SU(2)$ the first known manifold for which this propagation number is unequal to 2. Additionally, we consider the spectral localizer on $\SU(2)$, which we apply to unitary irreducible representations $\pi_j$ to investigate the ranges of $\kappa$ and $\rho$ at which the index is captured. 
Our results show a concrete operator system in which the index is recovered for a certain interval of $\kappa$ and $\rho$, but not in the limit as proposed in \cite{VanSuijlekom2024} Proposition 24.

\subsection*{Acknowledgements}
The author would like to thank Walter van Suijlekom for his support during the writing of the authors Master thesis, from which this work developed, and for the support received during the preparation of this work.

\subsection*{AI statement}
In parts of the research the LLMs ChatGPT and Google's search engine AI have been used as search engines to
explore the literature. The proof of Theorem \ref{thm:RJpropnr} was improved by Claude Opus 5.5, based on a previous result that $\prop(R_J) \leq 3$. This improvement, based on analyticity of matrix coefficients, could be generalized to Lie groups other than $\SU(2)$, and this was done by the author. The report itself, including all mathematical content, is written without generative AI. All mathematical arguments have been checked without the use of AI, and are the responsibility of the author.

\section{Preliminaries}
\subsection{Operator system invariants}
Let us start by recalling some concepts related to the theory of operator systems, while referring to \cite{Paulsen2003} for further details. \\
\\
An abstract unital operator system $(E,e)$ is a complex $*$-vector space with a matrix ordering, and an Archimedean order unit $e$. A linear map $\phi \colon E \to F$ between operator systems $(E,e)$ and $(F,f)$ is called unital if $\phi(e) = f$ and unital completely positive (ucp) if additionally for each $n \in \N$ it holds that the map $\phi^{(n)} \colon M_n(E) \to M_n(F)$ given by $\phi^{(n)}([x_{ij}]_{ij}) = [\phi(x_{ij})]_{ij}$ is positive. A complete order isomorphism is a ucp map $\phi \colon E \to F$ such that $\phi^{-1}$ is also ucp, and a complete order embedding is a map $\phi \colon E \to F$ such that $\phi \colon E \to \phi(E)$ is a complete order isomorphism. \\
Every closed linear subspace $S$ of a unital $C^*$-algebra $A$ such that $1 \in S$ and $S = S^*$ is an operator system with the inherited matrix ordering. An operator system of this type is called a concrete unital operator system, and it turns out that every unital operator system is of this form by the Choi-Effros theorem \cite{Choi1977}. As such, it is interesting to ask for the `smallest' $C^*$-algebra containing an operator system $E$, which is captured in the notion of $C^*$-envelope. 
For a given operator system $E$, the $C^*$-envelope $C^*_{env}(E)$ together with a complete order embedding $i_E \colon E \to C^*_{env}(E)$ is the unique $C^*$-algebra up to isomorphism such that $C^*(i_E(E)) = C^*_{env}(E)$ and for any other pair consisting of a $C^*$-algebra $A$ and complete order embedding $\lambda \colon E \to A$ such that $C^*(\lambda(E)) = A$, there is a surjective unital $*$-homomorphism $\pi \colon A \to C^*_{env}(E)$ making the diagram
\[\begin{tikzcd}
	E & A \\
	& {C^*_{env}(E)}
	\arrow["\lambda", hook, from=1-1, to=1-2]
	\arrow["i_E"', hook', from=1-1, to=2-2]
	\arrow["{\exists \pi}", two heads, from=1-2, to=2-2]
\end{tikzcd}\]
commute, and this $C^*$-envelope always exists \cite{Hamana1979}. \\
\\
The propagation number was introduced in \cite{ConnesvS2020} by Connes and Van Suijlekom as an invariant for operator systems, used for measuring how far an operator system is removed from being a $C^*$-algebra. Let $i_E(E)^{\circ n} \subseteq C^*_{env}(E)$ denote the norm-closure of the linear span of $n$-fold products of elements in $E$ for $n \in \N$. The propagation number $\prop(E)$ is defined to be the smallest $n$ such that $i_E(E)^{\circ n}$ is a $C^*$-algebra if such an $n$ exists, and $\prop(E) = \infty$ otherwise. The propagation number is invariant under stable equivalence of operator systems \cite{ConnesvS2020}, and hence invariant under a suitable notion of operator system Morita equivalence as well \cite{Koot2021, Gielen2023, Eleftherakis2026}.

\subsection{Representation theory of $\SU(2)$}
The representation theory of $\SU(2)$ is the main tool with which we carry out calculations. As such, we briefly review the relevant results, which are standard and can be found in many books on representation theory of Lie groups \cite{Hall2015,Duistermaat2000}. The unitary irreducible representations are labeled by half-integers $j \in \frac{1}{2}\N$, and the associated representations $(\pi_j, V_j)$ are $(2j+1)$-dimensional. The weight vectors $e_m$ of $V_j$ are labeled by $m \in \{-j, -j+1, \ldots, j\}$, which we will denote by $\jrel{m}{j}$ in the sequel, and these vectors form an orthonormal basis of $V_j$. We say that two half-integers have the same parity if their fractional parts are equal. By the Peter-Weyl theorem, we know that the normalized matrix coefficients
\begin{equation*}
	f^j_{m,m'}(g) := \sqrt{2j+1} \langle \pi_j(g) e_m, e_{m'} \rangle
\end{equation*}
with $j \in \frac{1}{2}\N$ and $\jrel{m,m'}{j}$ form an orthonormal basis of $L^2(\SU(2))$ under the Haar measure. These functions are furthermore dense in $C(\SU(2))$ in the supremum-norm.\\
\\
The product of two matrix coefficients is related to the tensor product of representations. For $\SU(2)$ specifically, the tensor product of irreducible representations decomposes via
\begin{equation*}
	V_{j_1} \tensor V_{j_2} \cong \bigoplus_{j_3 = |j_1 - j_2|}^{j_1 + j_2} V_{j_3}
\end{equation*}
and this isomorphism is realized by the Clebsch-Gordan intertwiner $C^{j_1, j_2}$. The coefficients
\begin{equation*}
	C^{j_1, j_2, j_3}_{m_1, m_2, m_3} := \langle C^{j_1, j_2} e_{m_1} \tensor e_{m_2}, e_{m_3} \rangle
\end{equation*}
are called the Clebsch-Gordan Coefficients (CGCs). All CGCs are real-valued when considering an appropriate phase convention \cite{Koornwinder1981,Varshalovich1988}. \begin{proposition}\label{prop:CGcond}
	The following conditions are necessary for the CGC $C^{j_1,j_2,j_3}_{m_1,m_2,m_3}$ to be nonvanishing:
	\begin{itemize}
	\item (Eligible weights) $\jrel{m_i}{j_i}$ for $i=1,2, 3$.
	\item (Triangle condition) $j_3 \in \{|j_1 - j_2|, |j_1 - j_2| + 1, \ldots, j_1 + j_2\}$.
	\item (Addition of weights) $m_1 + m_2 = m_3$.  
	\end{itemize}
\end{proposition}
With this, we can calculate the product of matrix coefficients to be
\begin{equation}\label{eq:matcoeffprod}
	f^{j_1}_{m_1, m_1'}f^{j_2}_{m_2, m_2'} = \sum_{j=|j_1 - j_2|}^{j_1 + j_2} \frac{\sqrt{2j_1 + 1}\sqrt{2j_2 + 1}}{\sqrt{2j + 1}} C^{j_1, j_2, j}_{m_1, m_2, m} C^{j_1, j_2, j}_{m_1', m_2', m'} f^{j}_{mm'}
\end{equation}
where $m = m_1 + m_2$ and $m' = m_1' + m_2'$. We view this equation as describing the multiplicative action of $f^{j_1}_{m_1, m_1'} \in C(\SU(2))$ on a basis function $f^{j_2}_{m_2, m_2'} \in L^2(\SU(2))$ in terms of basis functions $f^{j}_{m,m'} \in L^2(\SU(2))$.

\section{Peter-Weyl truncation}
In the remainder of this paper, we define $G := \SU(2)$ for notational convenience, unless specified otherwise. \\
\\
In this section we consider the \textbf{Peter-Weyl truncation} of $C(G)$, which is based purely on representation-theoretic aspects of $G$. For this we define a projection $Q_J$ at a cutoff level $J$.
\begin{definition}
	For $J \in \frac{1}{2}\N$, the \textbf{Peter-Weyl projection} at level $J$, $Q_J \colon L^2(G) \to L^2(G)$, is the orthogonal projection onto the subspace
	\begin{equation*}
		\bigoplus_{\substack{j=0 \\ j \in \frac{1}{2}\N}}^J V_j \tensor V_j^* \subset  \overline{\bigoplus_{\substack{j=0 \\ j \in \frac{1}{2}\N}}^\infty V_j \tensor V_j^*} \cong L^2(G). 
	\end{equation*} 
\end{definition}
Consider $C(G) \subseteq B(L^2(G))$ where a function is identified with its multiplicative action. Truncating this gives an operator system $R_J := Q_J C(G) Q_J$ represented on the Hilbert space $Q_JL^2(G)$.
To facilitate the study of these spaces, we define the following set.
\begin{definition}
	Let $J \in \frac{1}{2}\N$. Then we define $\J$ to be the set of half-integer triples $jmm'$ with $0 \leq j \leq J$, $\jrel{m}{j}$ and $\jrel{m'}{j}$.
\end{definition}
With this notation, it follows that $\{f^{j}_{m,m'}\}_{jmm' \in \II(J)}$ is an orthonormal basis of $Q_J$, and
\begin{equation}\label{eq:dimQJ}
	\dim Q_J L^2(G) = \dim \bigoplus_{\substack{j=0 \\ j \in \frac{1}{2}\N}}^J V_j \tensor V_j^* = \sum_{\substack{j=0 \\ j \in \frac{1}{2}\N}}^J (2j+1)^2 = \frac{1}{6}(2J+1)(2J + 2)(4J + 3)
\end{equation}
where the sum can be evaluated using a standard induction argument. As such, with $N = \dim Q_J L^2(G)$ one can view elements of $R_J$ as $N \times N$ matrices, where the rows and columns are indexed by $jmm' \in \II(J)$. To make this more explicit, we denote the truncation of $f \in C(G)$ by $\trunc[J]{f} := Q_Jf Q_J \in R_J$. In the rest of this section we assume $J$ is fixed, and simply write $\trunc{f}$. For components of this matrix, we use the notation $[f]^{j_1,m_1,m_1'}_{j_2,m_2,m_2'}$ for the entry on column $j_1m_1m_1' \in \II(J)$ and row $j_2 m_2 m_2' \in \II(J)$. That is,
\begin{equation*}
	[f]^{j_1,m_1,m_1'}_{j_2,m_2,m_2'} = \langle f f^{j_1}_{m_1, m_1'}, f^{j_2}_{m_2, m_2'} \rangle
\end{equation*}
where the inner product is taken in $Q_J L^2(G)$, or equivalently in $L^2(G)$.
For a normalized matrix coefficient $f^j_{m,m'}$ we will also use the notation $\trunc{jmm'} := \trunc{f^{j}_{m,m'}}$ for readability.
Writing out Equation \eqref{eq:matcoeffprod} in this notation gives
\begin{equation}\label{eq:matrixentriesRJ}
	[jmm']^{j_1,m_1,m_1'}_{j_2,m_2,m_2'} = \frac{\sqrt{2j + 1}\sqrt{2j_1 + 1}}{\sqrt{2j_2 + 1}} C^{j, j_1, j_2}_{m, m_1, m_2} C^{j, j_1, j_2}_{m', m_1', m_2'}
\end{equation}
for the coefficients of the matrix $\trunc{jmm'}$, which we will often use in computations. Multiplication of two matrices $A, B \in R_J$ takes the form
\begin{equation*}
	(AB)^{j_1 m_1 m_1'}_{j_2 m_2 m_2'} = \sum_{jmm' \in \J} A^{jmm'}_{j_2 m_2 m_2'} B^{j_1 m_1 m_1'}_{jmm'}
\end{equation*}
with this notation.

\setcounter{subsection}{0}
\subsection{Dimension of $R_J$}
\setcounter{subsection}{2}
Equipped with this notation, we will determine some operator system invariants of $R_J$, starting with the dimension.
\begin{theorem}\label{thm:RJdim}
	The matrices $\{\trunc{jmm'}\}_{\twoJ}$ form a basis of $R_J$.
	Thus,
    \begin{equation*}
        \dim R_J = \sum_{\substack{j=0\\j \in \frac{1}{2}\N}}^{2J} (2j+1)^2 = \frac{1}{6}(4J + 1)(4J+2)(8J + 3).
    \end{equation*}
\end{theorem}
Notice that for $j > 2J$ it holds that $[jmm'] = 0$.
Indeed, for such a $j$ the CGCs $C^{j, j_1, j_2}_{m, m_1, m_2}$ vanish for all $j_1 m_1 m_1', j_2 m_2 m_2' \in \J$ since the triangle inequality in Proposition \ref{prop:CGcond} is not satisfied: As $j-j_1 > J$ and $j_2 \leq J$ it holds that $j_2 \notin \{j-j_1, j - j_1 + 1\ldots, j + j_1 \}$. Thus, the truncation $\trunc{jmm'}$ certainly vanishes, which means $R_J$ is contained in the span of $\{[jmm']\}_{jmm' \in \II(2J)}$.\\
\\
The remaining difficulty in this proof consists then of showing that all $\trunc{jmm'}$ are nonzero and linearly independent. Although all $f^j_{mm'}$ are orthonormal in $L^2(G)$, and thus certainly linearly independent, it is a priori possible that the projection $Q_J$ creates dependencies when $f^{j}_{m,m'}$ is considered as an operator on $L^2(G)$. In order to prove that this does not happen, we first introduce a lemma about certain CGCs not vanishing.
\begin{lemma}\label{lem:CGNV}
	Let $C^{j_1, j_2, J}_{m_1, m_2, J}$ denote the CGC with $j_i \in \frac{1}{2}\N$ and $\jrel{m_i}{j_i}$ for $i=1,2$, and $J \in \frac{1}{2}\N$. Assume that the conditions in Proposition \ref{prop:CGcond} are satisfied. Then $C^{j_1, j_2, J}_{m_1, m_2, J} \neq 0$.
\end{lemma}
\begin{proof}
    Follow by Equation 35 in Section 8.5.2 of \cite{Varshalovich1988}.
\end{proof}

A stronger result follows as a consequence of classical symmetry relations between CGCs, such as described in \cite{Koornwinder1981,Varshalovich1988, Bohm2001}.

\begin{lemma}\label{lem:CGnonvanish}
	Let $C^{j_1,j_2,j_3}_{m_1,m_2,m_3}$ denote the CGC with $j_i \in \frac{1}{2}\N$ and $\jrel{m_i}{j_i}$ for $i=1,2,3$. Assume that the conditions in Proposition \ref{prop:CGcond} are satisfied, and additionally $|m_i| = j_i$ for at least one $i \in \{1,2,3\}$. Then $C^{j_1, j_2, j_3}_{m_1, m_2, m_3} \neq 0$.
\end{lemma}

We will explicitly use one such symmetry later, so we record it in a lemma.
\begin{proposition}[\cite{Varshalovich1988}, Equation 8.4.3]
	Let $j_i \in \frac{1}{2}\N$ and $\jrel{m_i}{j_i}$ for $i = 1,2,3$. Then
	\begin{equation}\label{eq:CGsym}
		C^{j_1, j_2, j_3}_{m_1, m_2, m_3} = (-1)^{j_2 + m_2} \sqrt{\frac{2j_3+1}{2j_1 + 1}} C^{j_2, j_3, j_1}_{-m_2, m_3, m_1}.
	\end{equation}
\end{proposition}
With this lemma we can prove Theorem \ref{thm:RJdim}.
\begin{proof}[Proof of Thm \ref{thm:RJdim}]
	We have already argued that $\trunc{jmm'} = 0$ for $j > 2J$ earlier, so it remains to show linear independence of $\trunc{jmm'}$ for $jmm' \in \II(2J)$.
    Suppose that there are $\lambda_{jmm'} \in \C$ for $jmm' \in \twoJ$ such that $\sum_{jmm' \in \twoJ} \lambda_{jmm'} \trunc{jmm'} = 0$. We prove that then necessarily $\lambda_{jmm'} = 0$ for all $jmm' \in \twoJ$. For two coordinates $j_1 m_1 m_1', j_2 m_2 m_2' \in \J$, the $(j_1 m_1 m_1', j_2 m_2 m_2')$'th coefficient gives
    \begin{equation*}
        \sum_{jmm' \in \II(2J)} \lambda_{jmm'} \trunc{jmm'}^{j_1 m_1 m_1'}_{j_2 m_2 m_2'}= \sqrt{\frac{2j_1 + 1}{2j_2 + 1}}\sum_{jmm' \in \twoJ} \sqrt{2j+1} \lambda_{jmm'} C^{j,j_1,j_2}_{m,m_1,m_2} C^{j,j_1,j_2}_{m',m_1',m_2'} = 0
    \end{equation*}
    using Equation \ref{eq:matrixentriesRJ}. Writing $\mu_{jmm'} := \sqrt{2j + 1} \lambda_{jmm'}$, which vanishes if and only if $\lambda_{jmm'}$ vanishes, gives
    \begin{equation}\label{eq:linindeppfcoords}
        \sum_{jmm' \in \twoJ} \mu_{jmm'} C^{j,j_1,j_2}_{m,m_1,m_2} C^{j,j_1,j_2}_{m',m_1', m_2'} = 0
    \end{equation}
    which holds for all $j_1 m_1 m_1', j_2 m_2 m_2' \in \II(J)$. We will make appropriate choices for these coordinates to deduce that all $\mu_{jmm'}$ vanish.\\
    \\
    First, consider $j_1 m_1 m_1' = 000$, and $j_2 m_2 m_2' \in \J$ arbitrary. Then Equation \ref{eq:linindeppfcoords} reads
    \begin{equation*}
        \sum_{jmm' \in \twoJ} \mu_{jmm'} C^{j,0,j_2}_{m,0,m_2}C^{j,0,j_2}_{m',0,m_2'} = 0.
    \end{equation*}
    By the triangle and weight condition the summands vanish if $j \neq j_2$, $m \neq m_2$ or if $m' \neq m_2'$. Thus, the equation reduces to $\mu_{j_2 m_2 m_2'} C^{j_2, 0, j_2}_{m_2, 0, m_2} C^{j_2,0,j_2}_{m_2,0,m_2} = 0$. By Lemma \ref{lem:CGnonvanish} these CGCs are nonzero (in fact, they are both equal to 1), hence we must necessarily have $\mu_{j_2 m_2 m_2'} = 0$. As we can do this for any $j_2 m_2 m_2' \in \J$ this yields that $\mu_{jmm'} = 0$ for all $jmm' \in \J$. \\
    \\
    For the remaining $\mu_{jmm'}$ with $j > J$, we proceed essentially by induction up to $2J$ in steps of one half. Suppose that $\mu_{jmm'} = 0$ for all $jmm' \in \II(L)$, with $L$ a half-integer satisfying $J \leq L < 2J$. Consider the triple $(L+\frac{1}{2}),l,l'$ for some $l,l'$ satisfying $\jrel{l,l'}{L+\frac{1}{2}}$. Then we show that $\mu_{(L+\frac{1}{2}) ll'} = 0$. \\
	\\
	Define $s = L + \frac{1}{2} -J$, $r = -\sign(l) s$, $r' = -\sign(l') s$ and take the two coordinates $j_1 m_1 m_1' = s r r'$ and $j_2 m_2 m_2' = J (l+r) (l'+r')$.
	These are indeed valid coordinates for $Q_JL^2(G)$: For the triple $srr'$, we know that $0 < s \leq J$ by the requirements on $L$, so that $srr' \in \J$. Furthermore, since $l+r = \sign(l)(|l| - L -\frac{1}{2} + J)$, we can see that $l+r$ and $J$ have the same parity (as $|l| - L - \frac{1}{2} \in \Z$), and furthermore $|l+r| \leq J$ since $|l| - L - \frac{1}{2} \leq 0$. An identical argument works for $l' + r'$, so indeed $J (l+r) (l' + r') \in \J$.\\
	\\
	Now consider Equation \eqref{eq:linindeppfcoords} with these $j_1 m_1 m_1', j_2 m_2 m_2'$. For the CGCs $C^{j, s, J}_{m, r, l+r}$ and $C^{j, s, J}_{m', r', l'+r'}$ to be nonvanishing for some $jmm' \in \twoJ$, it is necessary that $J \in \{|j-L-\frac{1}{2} +J|, |j-L - \frac{1}{2} + J| + 1, \ldots, j+L+\frac{1}{2} -J\}$ and $m = l$, $m' = l'$, by the triangle condition and magnetic moments. This implies that $j \leq L + \frac{1}{2}$ is necessary for this not to vanish, as otherwise $|j-L-\frac{1}{2}+J| > J$. Thus, Equation \eqref{eq:linindeppfcoords} reduces at least
	to
	\begin{equation*}
		\sum_{j = \sigma}^{L+\frac{1}{2}} \mu_{jll'} C^{j, s, J}_{l,r,(l+r)} C^{j,s,J}_{l',r',(l'+r')} = 0  
	\end{equation*}
	where the summation starts at $\sigma = L + \frac{1}{2} \text{ (mod }1)$, $\sigma \in \{0, \frac{1}{2}\}$, and takes integer steps. However, since we have assumed that $\mu_{jmm'}$ vanishes for all $j \leq L$, the sum reduces further to $\mu_{(L+\frac{1}{2})ll'} C^{L+\frac{1}{2}, s, J}_{l,r,(l+r)} C^{L+\frac{1}{2},s,J}_{l',r',(l'+r')} = 0$. Since $r,r' \in \{-s,s\}$, the CGCs in this equation do not vanish by Lemma \ref{lem:CGnonvanish}, from which we find that $\mu_{(L+\frac{1}{2})ll'} = 0$ and hence $\lambda_{(L+\frac{1}{2})ll'} = 0$. By induction we conclude that $\lambda_{jmm'} = 0$ for all $jmm' \in \twoJ$, which shows that all $[jmm']$ are linearly independent in $R_J$ and hence constitute a basis. The total dimension of $R_J$ can be calculated in the same way as $\dim Q_J L^2(G)$.
\end{proof} 
\setcounter{subsection}{1}
\subsection{$C^*$-envelope and propagation number}
\setcounter{subsection}{6}
In this section, we determine the invariants $C^*_{env}(R_J)$ and $\prop(R_J)$. In the case that $S \subseteq M_n(\C)$ is an operator system represented on a finite-dimensional Hilbert space, as is the case with $R_J$, the definition of $C^*$-envelope reduces to the following (\cite{Koot2021}, Prop. 1.37).
\begin{definition}
	Let $S \subseteq M_n(\C)$ be an operator system. Then the \textbf{$C^*$-envelope} of $S$, denoted $C^*_{env}(S)$, is the smallest $C^*$-subalgebra of $M_n(\C)$ such that $S \subseteq C^*_{env}(S)$.
\end{definition}
It should be noted that all finite-dimensional $C^*$-algebras are of the form $\bigoplus_{i = 1}^N M_{n_i}(\C)$, up to isomorphism. Thus, the $C^*$-envelope may not necessarily be $M_n(\C)$ itself, but it will always be isomorphic to some combination of $M_{n_i}(\C)$ blocks with $\sum_{i=1}^N n_i \leq n$. Furthermore, since $M_n(\C)$ is finite-dimensional, the propagation number will always be finite with an upper bound depending on $n$. \\
\\
The proof given in this section works for the Peter-Weyl truncation of more general Lie groups, namely compact connected Lie groups. An ingredient in this proof is \textbf{real-analyticity} of matrix coefficients. We briefly recall what it means for (functions on) a manifold to be real-analytic, using Sections 1.9 and 2.1 of the book by Lewis \cite{Lewis2023}.
\begin{definition}
    Let $U \subseteq \R^n$ be open. A function $f : U \to \R^m$ is called real-analytic if $f$ equals its convergent power series near every point. That is, for all $x_0 \in U$ there is a neighbourhood $V$ of $x_0$ such that
\begin{equation*}
    f(x) = \sum_{I \in \N^n} \frac{D^I f(x_0)}{I!}(x-x_0)^I
\end{equation*}
    on $V$ where $I = (\alpha_1, \ldots, \alpha_n)$ is some multi-index and $I! = \prod_{i=1}^n \alpha_i!$.
\end{definition}

\begin{definition}
    A real-analytic manifold $M$ is a smooth manifold such that all transition functions are real-analytic. A function $f : M \to N$ is real-analytic if it is real-analytic on charts.
\end{definition}
Real-analytic maps admit an \textbf{identity theorem}, which is perhaps more familiar from complex geometry. We will not state this in full generality, but only in the context as needed for our proof.
\begin{proposition}[\cite{Lewis2023}, Sec. 2.1.2]
    Let $M$ be a connected real-analytic manifold, and $f : M \to \R$ be a real-analytic function. If there is an open nonempty subset $U \subseteq M$ such that $f|_{U} = 0$, then $f = 0$ on all of $M$.
\end{proposition}

For Lie groups we have the following.
\begin{proposition}[\cite{Duistermaat2000}, Sec. 1.6 + Cor 1.10.9]
    Every Lie group has a unique real-analytic structure such that the product and inverse maps are real-analytic. Furthermore, every Lie group homomorphism $\phi : G \to G'$ is real-analytic with respect to these structures.
\end{proposition}
We will use the consequence that matrix coefficients are analytic. For $G = \SU(2)$ this is clear as matrix coefficients can be written as (Jacobi) polynomials \cite{Koornwinder1981}.\\
\\
To describe the Peter-Weyl truncation of general Lie groups, we briefly introduce some notation.
Let $G$ be a connected compact Lie group. We can decompose $L^2(G)$ with the Haar measure using the Peter-Weyl decomposition
\begin{equation*}
    L^2(G) \cong \bigoplus_{\pi \in \widehat G} V_\pi \tensor V_{\pi}^*
\end{equation*}
where $\widehat G$ denotes the set of unitary irreducible representations of $G$. Consider a finite-rank orthogonal projection $P_\Lambda$ with
\begin{equation*}
    P_\Lambda \bigoplus_{\pi \in \widehat G} V_\pi \tensor V_{\pi}^* = \bigoplus_{\pi \in \Lambda} V_\pi \tensor V_{\pi}^*
\end{equation*}
with $\Lambda\subset \widehat G$ a \textbf{finite} set of irreps. Define $R_\Lambda = P_\Lambda C(G)P_{\Lambda}$, and write $[f] := P_{\Lambda} f P_{\Lambda} \in R_\Lambda$ for the truncation of $f \in C(G)$.

\begin{theorem}\label{thm:RJpropnr}
    Let $G$ be a compact connected Lie group and $R_{\Lambda}$ defined as above. It holds that $C^*_{env}(R_{\Lambda}) = M_{N}(\C)$ where $N = \dim P_{\Lambda} L^2(G)$. Furthermore, $\mathrm{prop}(R_{\Lambda}) \leq 2$. 
\end{theorem}

In the proof, we use some ideas similar to those in the paper \cite{Leimbach24}. In particular, we make use of an integral kernel $K_{\Lambda}$ encoding the spectral truncation, which was used in that paper to prove convergence of state spaces.

\begin{proof}
    We aim to prove that $R_{\Lambda}^{\circ 2} = M_N(\C)$. For this we use the Hilbert-Schmidt inner product on $M_N(\C)$ given by $\langle A, B \rangle_{HS} = \Tr(AB^*)$ to show that the orthogonal complement $(R_{\Lambda}^{\circ 2})^{\perp} \subseteq M_N(\C)$ equals zero. Then $R_{\Lambda}^{\circ 2} = (R_{\Lambda}^{\circ 2})^{\perp \perp} = M_N(\C)$ proving the claim. \\
    \\
    To assist in this, we introduce some notation.
    Define the normalized matrix coefficients $$f^\pi_{v,w}(g) = \sqrt{\dim V_{\pi}}\langle \pi(g) v, w \rangle$$
    for $\pi \in \widehat G$, $v,w \in V_\pi$. Consider an orthonormal basis $\{e^{\pi}_{k}\}_{k \in I_\pi}$ of $V_{\pi}$, and define the finite index set
    \begin{equation*}
        I = \{ (\pi, e^\pi_k, e^\pi_l) \mid \pi \in \Lambda; k,l \in I_{\pi} \}.
    \end{equation*}
    For $i \in I$ define the matrix coefficient $f_i := f^{\pi}_{v,w}$. In the following it will only be relevant that $\{f_i\}_{i \in I}$ is an orthonormal basis of $P_\Lambda L^2(G)$, so we suppress the representation-theoretic origin in our notation.
    Then by definition
    \begin{equation}
        (P_{\Lambda} g)(x) = \sum_{i \in I} \langle f_i, g \rangle_{L^2(G)} f_i(x) = \sum_{i \in I} \int_G f_i(x) f_i^*(y) g(y) \; dy
    \end{equation}
    for $g \in L^2(G)$, so by defining the integral kernel $K_{\Lambda}(x,y) = \sum_{i \in I} f_i(x) f_i^*(y)$ we may write $(P_{\Lambda} g)(x) = \int_G K_{\Lambda}(x,y)g(y) \; dy$. \\
    \\
    Consider now $g,h \in C(G)$ and $B \in M_N(\C)$ arbitrary, and assume that $\langle [g][h], B \rangle_{HS} = 0$. We aim to show that in fact $B = 0$. Since $\{f_i\}_{i \in I}$ forms a basis of $P_{\Lambda} L^2(G)$, we may write $B v = \sum_{i,j \in I} b_{ij} \langle v, f_i \rangle f_j $ so that $B^*v = \sum_{i,j \in I} b_{ji}^*\langle v, f_i \rangle f_j $ for $v \in P_JL^2(G)$. Extending $B^*$ with zeroes to act on all of $L^2(G)$ we may write
    \begin{equation}
        (B^* g)(x) = \sum_{i, j \in I} b^*_{ji} \int_G f_j(x) f_i^*(y) g(y) \; dy = \int_G K_B(x,y) g(y) \; dy
    \end{equation}
    for $g \in L^2(G)$, where $K_B(x,y) = \sum_{i,j \in I} b_{ji}^* f_j(x) f_i^*(y) $ similarly as for $P_{\Lambda}$. \\
    \\
    With this writing, we can simplify the trace as
    \begin{align*}
        \Tr([g][h]B^*) &= \Tr(P_{\Lambda} g P_{\Lambda} P_{\Lambda} h P_{\Lambda} B^*) \\
        & = \Tr(g P_{\Lambda}^2 h P_{\Lambda} B^* P_{\Lambda}) \\
        & = \Tr(g P_{\Lambda} h B^*)
    \end{align*}
    where we have used the cyclic property of the trace together with $P_{\Lambda} B^* P_{\Lambda} = B^*$ and $P_{\Lambda}^2 = P_{\Lambda}$. Here the functions $g$ and $h$ are identified with their multiplication operators. Evaluating this trace with the basis $\{f_i\}_{i \in I}$ gives
    \begin{align*}
        \Tr(g P_{\Lambda} h B^*) &= \sum_{i \in I} \int_G f_i^*(x) (gP_{\Lambda} h B^* f_i)(x) \; dx \\
        &= \sum_{i \in I} \int_G f^*_i(x) g(x) \left (\int_G K_\Lambda(x,y) (hB^* f_i)(y) dy \right ) dx \\
        &= \sum_{i \in I} \int_G f^*_i(x) g(x) \left (\int_G K_\Lambda(x,y) h(y) \sum_{j \in I} b_{ji}^* f_j(y) dy \right ) dx \\
        &= \int_{G \times G} g(x)h(y) K_{\Lambda}(x,y) \left (\sum_{i,j \in I} b_{ji}^*f_j(y)f^*_i(x) \right ) dx dy \\
        &= \int_{G \times G} g(x)h(y) K_{\Lambda}(x,y)K_B(x,y) \; dx dy
    \end{align*}
    where we have used the integral kernel formulation of $P_{\Lambda}$ and $B^*$, together with $B^* f_i = \sum_{j \in I} b_{ji}^* f_j$ and Fubini's theorem. By assumption, this trace vanishes so that
    \begin{equation}
        \int_{G \times G} g(x)h(y) K(x,y) \; dx dy = 0
    \end{equation}
    for all functions $g,h \in C(G)$, where $K(x,y) := K_\Lambda (x,y) K_B(x,y)$. By the Stone-Weierstrass theorem, the set of product functions $\{g(x)h(y) \mid g,h \in C(G)\}$ is dense in $C(G \times G)$, and hence in $L^2(G \times G)$. Since $K$ is continuous and in the orthogonal complement of a dense set, we find that $K(x,y) = 0$ everywhere.\\
    \\
    In fact, since the matrix coefficients $f_i$ are analytic, the functions $K_\Lambda$, $K_B$ and $K$ are all analytic.
    Since $K(x,y) = 0$, the function $K_B(x,y)$ vanishes on the open set $\{K_{\Lambda} \neq 0 \}$, which is nonempty as $K(e,e) = \sum_{i \in I} |f_i(e)|^2 = \sum_{\pi \in \Lambda} (\dim V_\pi) \neq 0$. A real-analytic function vanishing on a nonempty open set vanishes everywhere on the connected component. Hence $K_B(x,y) = 0$ for all $x,y \in G$, which implies that $B^*g = 0$ for all $g \in L^2(G)$, so that $B = 0$ as desired.
\end{proof}
Since $\SU(2)$ is connected, we can apply the theorem to operator systems $R_J$ and find the following.

\begin{corollary}
    Let $R_J$ be the Peter-Weyl truncations of $G = \SU(2)$. Then $C^*_{env}(R_J) = M_N(\C)$ where $N = \dim P_J L^2(G)$. For $J \geq \frac{1}{2}$ it holds that $\prop(R_J) = 2$, while for $J = 0$ we have $R_0 = \C = M_1(\C)$ so that $\prop(R_0) = 1$.
\end{corollary}
It is clear that $\prop(R_J) \neq 1$ for $J \geq \frac{1}{2}$ by comparing dimensions of $R_J$ in Theorem \ref{thm:RJdim} and of $Q_J L^2(G)$ as determined in Equation \eqref{eq:dimQJ}, establishing $\prop(R_J) = 2$.

\begin{remark}
    Although this theorem is phrased for connected compact Lie groups, the essential ingredients for the proof are not necessarily representation-theoretic. Instead, for a connected compact real-analytic manifold $M$, all that is required is an orthonormal basis $\{f_i\}_{i \in I}$ of $L^2(M)$ consisting of \textbf{real-analytic} functions. The theorem then holds for finite projections onto this basis. Here $K_\Lambda(x,x) = \sum_{i \in I} |f_i(x)|^2$ cannot be zero for all $x \in M$, which means $\{K_{\Lambda} \neq 0\}$ is nonempty and allows for the analyticity argument to go through. 
\end{remark}

Analyticity connects with another type of truncation, based on the Laplacian. Under suitable conditions the eigenfunctions of the Laplacian are real-analytic up to the boundary, which is a classical result in the regularity theory of elliptic operators \cite{Morrey1957}. On a real-analytic compact connected Riemannian manifold $(M,g)$, where also $g$ is assumed to be real-analytic, the eigenfunctions of the Laplacian (i.e. the Laplace-Beltrami operator) form an orthonormal basis of $L^2(M)$ consisting of analytic functions, so that one can consider truncations \cite{CHARRON2025110792}. In this context, the theorem can be reinterpreted in the following way.
\begin{corollary}
    Consider a compact connected real-analytic Riemannian manifold $(M,g)$ and the projection $P_\Lambda$ on $L^2(M)$ given by a spectral projection of the Laplacian on a finite amount of eigenvalues, $P_{\Lambda} = \chi_{[0,\Lambda]}(\Delta_{M})$. Then $R_{\Lambda} = P_\Lambda C(M) P_{\Lambda}$ has $C^*_{env}(R_{\Lambda}) = M_N(\C)$ and $\prop(R_{\Lambda}) \leq 2$, where $N = \dim P_{\Lambda}L^2(M)$. 
\end{corollary}
Examples of this setting include the spheres $S^n$, where the eigenfunctions are given by higher-dimensional spherical harmonics. These are polynomials when written in suitable coordinates, and are thus certainly analytic \cite{Efthimiou2014}. A possible interpretation of this result is that the spectrum of the Laplacian does not carry enough information to make the propagation number of truncated systems a meaningful way of distinguishing different manifolds, provided they are sufficiently nice. 

\begin{remark}
    In the case of vector-valued functions in $C(M,\C^n)$, the kernels $K_J, K_B$ in the proof of Theorem \ref{thm:RJpropnr} become matrix-valued. The calculation performed then gives that the trace of $K_J \cdot K_B$ is identically zero, but this is not enough to conclude that all of $K_B$ is zero. Thus, the spectral truncation corresponding to the Dirac operator may have a propagation number larger than 2.
\end{remark}

\section{Spectral truncations of $\SU(2)$}
In the previous section we have studied Peter-Weyl truncations of $G$, which are based solely on its representation theory. In this section, we instead consider a more geometric point of view by truncating the spectral triple $(C^{\infty}(G), L^2(S), D)$ arising from a spinor bundle and Dirac operator on $G = \SU(2)$. For $\SU(2)$ specifically, an analysis of the Dirac operator and spinor bundle has been performed by Homma \cite{Homma2000} based on the exceptional isomorphisms $\Spin(4) \cong \SU(2) \times \SU(2)$ and $\Spin(3) \cong \SU(2)$. The rank 2 spinor bundle $S \to G$ can be trivialized using bundle map $\Delta^+ \colon S \xrightarrow{\sim} G \times \C^2$ as described in \textit{loc. cit.} which identifies $L^2(S) \cong L^2(G) \tensor \C^2$. The Dirac operator in this trivialization is given by 
\begin{equation*}
    D = \frac{3}{2} + \sigma_1 Z_1 + \sigma_2 Z_2 + \sigma_3 Z_3
\end{equation*}
where
\begin{align*}
    \sigma_1 = \begin{pmatrix}
    i & 0 \\ 0 & -i
    \end{pmatrix}, \phantom{==} \sigma_2 = \begin{pmatrix}
    0 & 1 \\ -1 & 0
    \end{pmatrix}, \phantom{==}
    \sigma_3 = \begin{pmatrix}
    0 & i \\ i & 0
    \end{pmatrix}, 
\end{align*}
are Pauli matrices in $\mathfrak{su}(2)$ satisfying $\sigma_i \sigma_j + \sigma_j \sigma_i = -2\delta_{ij}$. The $Z_i$ are their associated right-invariant vector fields defined by
\begin{equation*}
    (Z_i f)(g) = \frac{d}{dt} f((\exp t \sigma_i) g) \mid_{t = 0}
\end{equation*}
for $f \in C^{\infty}(G)$. The eigenspinors of the Dirac operator can be found explicitly in this scenario, which is described in the following theorem.
\begin{theorem}[\cite{Homma2000}]\label{thm:eigenspinors}
    The eigenvalues of the spin-Dirac operator $D$ on $\SU(2)$ are given by $\pm(\frac{3}{2} + 2j)$ for $j \in \frac{1}{2}\N$. Let $E_{\pm j}$ denote the eigenspace corresponding to the eigenvalue $\pm(\frac{3}{2}+2j)$. Then $\dim E_{\pm j} = (2j+1)(2j+2)$, and orthonormal bases of $E_{+j}$ and $E_{-j}$ are given by
    $\{\psi^j_{m,m'} \mid  -j \leq m \leq j, \; -j-1 \leq m' \leq j\}$ and $\{\psi^{-j}_{m,m'}  \mid  -j \leq m \leq j+1, \; -j \leq m' \leq j\}$ respectively, with $m$ and $m'$ having the same parity as $j$, where
    \begin{align*}
        \psi^{+j}_{m,m'} &= \frac{1}{\sqrt{2j+1}} \begin{pmatrix}
        -\sqrt{j-m'} f^{j}_{m,m'+1} \\
        \sqrt{j+1+m'} f^{j}_{m,m'}
        \end{pmatrix} \phantom{=}&& (-j \leq m \leq j, \; -j-1 \leq m' \leq j) \\
        \psi^{-j}_{m,m'} &= \frac{1}{\sqrt{2j+2}} \begin{pmatrix}
        \sqrt{j+1+m'} f^{j+\frac{1}{2}}_{m-\frac{1}{2}, m' + \frac{1}{2}} \\
        \sqrt{j+1-m'} f^{j+\frac{1}{2}}_{m-\frac{1}{2},m'-\frac{1}{2}}
        \end{pmatrix} \phantom{=}&& (-j \leq m \leq j+1, \; -j \leq m' \leq j).
    \end{align*}
	The $+$ and $-$ in $E_{+j}$ and $E_{-j}$ should be thought of as a label instead of an operation on $j$, considering that $E_{+0} \neq E_{-0}$. 
    A decomposition of $L^2(S)$ is given by
    \begin{equation*}
        L^2(S) = \overline{\bigoplus_{j\in\frac{1}{2}\N} E_{-j} \oplus E_{+j}}
    \end{equation*}
    where the completion is taken in the $L^2(S)$-norm.
\end{theorem}
The eigenspaces $E_{+j}$ and $E_{-j}$ can be viewed in terms of the representation theory of $\SU(2)$. It holds that $E_{+j} = V_{j+\frac{1}{2}} \tensor V_j \tensor \Hom_G(V_{j + \frac{1}{2}} \tensor V_j, V_{\frac{1}{2}})$, where $\Hom_G(V,W)$ denotes the vector space of intertwiners from $V$ to $W$, which has dimension one in this case. An element $v \tensor w \tensor A$ can be interpreted as a function $G \to \C^2$ by using $V_{\frac{1}{2}} \cong \C^2$ and considering $\psi(g) = A(v \tensor \pi_{j}(g)w)$. Similarly, $E_{-j} = V_{j} \tensor V_{j+\frac{1}{2}} \tensor \Hom_G(V_{j} \tensor V_{j + \frac{1}{2}}, V_{\frac{1}{2}})$ where some $v \tensor w \tensor A$ defines a function via $\psi(g) = A(v \tensor \pi_{j+\frac{1}{2}}(g)w)$. \\
\\
For $J \in \frac{1}{2}\N$ the (Dirac) spectral projection is defined by
$$P_J = \chi_{[-(2J+\frac{3}{2}), 2J+\frac{3}{2}]}(D) \colon L^2(S) \to L^2(S)$$ where $\chi_A \colon \R \to [0,1]$ is the characteristic function of the set $A$. Since the eigenvalues of $D$ are precisely $\pm (2j + \frac{3}{2})$ for $j \in \frac{1}{2}\N$, we have that
\begin{equation*}
    P_J L^2(S) = \bigoplus_{\substack{j = 0 \\ j \in \frac{1}{2}\N}}^J E_{-j} \oplus E_j,
\end{equation*} 
and hence $$\dim P_J L^2(S) = 2\sum_{\substack{j=0 \\ j \in \frac{1}{2}\N}}^{J} (2j+1)(2j+2) = \frac{2}{3}(2J+1)(2J+2)(2J+3).$$
Define the \textbf{spectral truncation} of $C^\infty(G)$ acting on $L^2(S)$ to be $S_J := P_J C^\infty(G) P_J$ acting on $P_J L^2(S)$. For an $f \in C^\infty(G)$, we denote the spectral truncation of $f$ at level $J$ by $\spintrunc[J]{f}$ and for a matrix coefficient $f^{j}_{m,m'}$ we write $\spintrunc[J]{jmm'} := \spintrunc[J]{f^{j}_{m,m'}}$. In the rest of this chapter, we fix $J$ as the truncation level and use $\spintrunc{jmm'}$ for the corresponding spectrally truncated matrices.\\
\\
We furthermore use the notation $\II_S(J)$ for the set of labels for basis spinors in $P_JL^2(S)$, so that elements of $S_J$ can be seen as matrices with rows and columns indexed by $\II_S(J)$. To make this more precise, we write $\jrelplus{m,m'}{j}$ for a triple $jmm'$ such that $\psi^{+j}_{m,m'}$ is a sensible positive spinor, for a half-integer $j \in \frac{1}{2}\N$. That is, when writing $\jrelplus{m,m'}{j}$ we mean that $m \in \{-j, -j+1, \ldots, j\}$ and $m' \in \{-j-1, -j, \ldots, j\}$. Similarly, when we write $\jrelminus{m,m'}{j}$ we mean that $\psi^{-j}_{m,m'}$ is a sensible negative spinor, which is equivalent to $m \in \{-j,-j+1, \ldots, j+1\}$ and $m' \in \{-j, -j+1, \ldots, j\}$. 
With this notation, $\II_S(J)$ is the set containing the half-integer triples $(+j)mm'$ with $0 \leq j \leq J$ and $\jrelplus{m,m'}{j}$, and the half-integer triples $(-j)mm'$ with $0 \leq j \leq J$ and $\jrelminus{m,m'}{j}$. Since $\II_S(J)$ contains both positive and negative triples, when we write $jmm \in \II_S(J)$ it is possible that $j$ is positive or negative. When we want to stress that we are considering either positive or negative eigenspinors, we may write $+jmm$ or $-jmm$ respectively. Entries of the matrix $\spintrunc{jmm}$ are   \begin{equation*}
	\spintrunc{jmm'}^{j_1 m_1 m_1'}_{j_2 m_2 m_2'} = \langle f^{j}_{m,m'} \psi^{j_1}_{m_1, m_1'}, \psi^{j_2}_{m_2, m_2'} \rangle
\end{equation*}
where $j_1 m_1 m_1', \; j_2 m_2 m_2' \in \II_S(J)$. \\
\\
We furthermore say that the \textbf{degree} of $\psi^{+j}_{m,m'}$ is $j$, while the degree of $\psi^{-j}_{m,m'}$ is $j+\frac{1}{2}$, corresponding to the labels of irreps which appear in their definitions. Alternatively, we may treat the $+$ or $-$ sign as a label, writing $\deg(j) = j$ and $\deg(-j) = j+\frac{1}{2}$ for $j \geq 0$. As the definition of $\psi^{-j}_{m,m'}$ spinors contains several occurrences of $+\frac{1}{2}$ and $-\frac{1}{2}$, we introduce the following notation for convenience.
\begin{definition}
    For a half-integer $k \in \frac{1}{2}\N$, we define $k^+ := k + \frac{1}{2}$ and $k^- := k - \frac{1}{2}$.
\end{definition}
For the Peter-Weyl truncation, we combined Equation \eqref{eq:matrixentriesRJ} and the conditions Proposition \ref{prop:CGcond} to determine when coefficients in a matrix $\trunc{jmm'}$ can be nonzero. The equivalent for $\spintrunc{jmm'}$ is described in the following lemma.
\begin{lemma}
    Assume that an inner product of the form
    \begin{equation*}
        \langle f^{j}_{m, m'} \psi^{j_1}_{m_1, m_1'}, \psi^{j_2}_{m_2, m_2'} \rangle
    \end{equation*}
    is nonvanishing, with suitable labels $j,m,m',j_i, m_i, m_i'$ for $i = 1,2$, where $j_1$ and $j_2$ can be positive or negative. Then it holds that $$\deg(j_2) \in \{|j - \deg(j_1)|, |j - \deg(j_1)| + 1, \ldots, j + \deg(j_1)\}. $$
    Furthermore, the magnetic moments add together based on the signs of $j_1$ and $j_2$ according to Table \ref{tab:magneticmomentsadd}.
    \begin{table}[H]
    \centering
    \begin{tabular}{|c|c|c|c|}
    \hline
    $\sign(j_1)$ & $\sign(j_2)$ & $m$ addition rule            & $m'$ addition rule              \\ \hline
    +            & +            & $m + m_1 = m_2$   & $m' + m_1' = m_2'$   \\ \hline
    +            & -            & $m + m_1 = m_2^-$ & $m' + m_1' = m_2'^-$ \\ \hline
    -            & +            & $m + m_1^- = m_2$ & $m' + m_1'^- = m_2'$ \\ \hline
    -            & -            & $m + m_1 = m_2$   & $m' + m_1' = m_2'$   \\ \hline
    \end{tabular}
    \caption{Magnetic moment conditions for spinors depending on the sign of the spinors.}    \label{tab:magneticmomentsadd}
    \end{table}
\end{lemma}
\begin{proof}
    This follows from Proposition \ref{prop:CGcond} and the explicit writing of positive and negative eigenspinors.
\end{proof}
Since the following spinors and inner products involving them will come up often, we record some facts in a separate lemma.
\begin{lemma}\label{lem:maxspinorsinprod}
For $j \in \frac{1}{2}\N$ and $\jrel{m}{j}$ it holds that
\begin{align}\label{eq:specialspinors}
    \psi^{+j}_{m,j} = \begin{pmatrix}
        0 \\ f^{j}_{m,j}
    \end{pmatrix}
    &\text{ and } \: \psi^{+j}_{m,-j-1} = -\begin{pmatrix}
    f^j_{m,-j} \\ 0
    \end{pmatrix}.
\end{align}
Consequently, we have the following formulas.
\begin{align*}
    \langle f^{j_1}_{m_1, m_1'} \psi^{+j_2}_{m_2, m_2}, \psi^{+j}_{m, j} \rangle &= \sqrt{\frac{2j_1+1}{2j+1}}\sqrt{j_2+1+m_2'} C^{j_1, j_2, j}_{m_1, m_2, m} C^{j_1, j_2, j}_{m_1', m_2', j} \\
    \langle f^{j_1}_{m_1, m_1'} \psi^{-j_2}_{m_2, m_2}, \psi^{+j}_{m,j} \rangle &= \sqrt{\frac{2j_1 + 1}{2j+1}} \sqrt{j_2 + 1 - m_2'} C^{j_1, j_2^+, j}_{m_1, m_2^-, m} C^{j_1, j_2^+, j}_{m_1', m_2'^-, j} \\
    \langle f^{j_1}_{m_1, m_1'} \psi^{+j_2}_{m_2, m_2'}, \psi^{+j}_{m,-j-1} \rangle &= \sqrt{\frac{2j_1+1}{2j+1}}\sqrt{j_2-m_2'} C^{j_1, j_2, j}_{m_1, m_2, m} C^{j_1, j_2, j}_{m_2', m_2'+1, -j} \\
    \langle f^{j_1}_{m_1, m_1'} \psi^{-j_2}_{m_2, m_2'}, \psi^{+j}_{m,-j-1} \rangle &= -\sqrt{\frac{2j_1+1}{2j+1}} \sqrt{j_2+1+m_2'} C^{j_1, j_2^+, j}_{m_1, m_2^-, m} C^{j_1, j_2^+, j}_{m_1', m_2'^+, -j}
\end{align*}
\end{lemma}
\begin{proof}
    Equation \eqref{eq:specialspinors} follows from substituting $m' = j$ and $m' = -j-1$ into the definition of the spinors from Theorem \ref{thm:eigenspinors}, while the inner products can be calculated using Equation \eqref{eq:matcoeffprod}.
\end{proof}

We have the following result for the dimension of $S_J$.
\begin{theorem}\label{thm:SJdim}
    The matrices $\{\spintrunc{jmm'}\}_{jmm \in \II(2J+\frac{1}{2})}$ form a basis of $S_J$. As such, the dimension of $S_J$ is given by
    \begin{equation*}
        \dim S_J = \sum_{\substack{j=0 \\ j \in \frac{1}{2}\N}}^{2J+\frac{1}{2}} (2j+1)^2 = \frac{1}{6}(4J + 2)(4J + 3)(8J + 5)
    \end{equation*}
\end{theorem}
This theorem should be compared to Theorem \ref{thm:RJdim} for the Peter-Weyl truncation. Parts of the proof of Theorem \ref{thm:SJdim} are similar to that of Theorem \ref{thm:RJdim}. In particular, proving that the matrices $\spintrunc{jmm'}$ are linearly independent for $jmm' \in \II(2J + \frac{1}{2})$ is essentially proven by reducing to the Peter-Weyl case, and the fact that $\spintrunc{jmm'} = 0$ for $j > 2J + 1$ follows the same logic as well. The new part of the proof is that for $j = 2J + 1$ we have $\spintrunc{jmm'} = 0$. This is a priori surprising, since $\psi^{-J}_{m,m'}$ involves matrix coefficients of $\pi_{J+\frac{1}{2}}$, and thus $\langle f^{2J+1}_{m,m'} \psi^{-J}_{m_1,m_1'}, \psi^{-J}_{m_2, m_2'} \rangle$ could potentially yield an expression with nonzero CGCs. However, it turns out that the relevant CGCs always cancel each other out in this scenario, which we will prove first.
\begin{lemma}[\cite{Varshalovich1988} Eq. 8.5.3.]
    For $j_1, j_2 \in \frac{1}{2}\N$ and $\jrel{m_1}{j_1}, \; \jrel{m_2}{j_2}$, the CGC corresponding to $j_1 + j_2$ is given explicitly by
    \begin{equation}\label{eq:CGj1plusj2}
        C^{j_1, j_2, j_1+j_2}_{m_1, m_2, m_1+m_2} = \sqrt{\frac{(2j_1)!(2j_2)! (j_1+j_2+m_1+m_2)!(j_1+j_2-m_1-m_2)!}{(2j_1+2j_2)!(j_1+m_1)!(j_1-m_1)!(j_1-m_1)!(j_2+m_2)!(j_2-m_2)!}}.
    \end{equation}
\end{lemma}
\begin{lemma}\label{lem:CGDiffzero}
    The following identity of CGCs holds for all $j \in \frac{1}{2}\N$, $\jrel{m_1}{j}$ and $\jrel{m_2}{j}$
    \begin{equation*}
        \sqrt{j-m_1} \sqrt{j+m_2} C^{j,j,2j}_{m_1,m_2, m_1 + m_2} - \sqrt{j+m_1 + 1}\sqrt{j-m_2 + 1} C^{j,j,2j}_{m_1+1,m_2-1, m_1+m_2} = 0.
    \end{equation*}
\end{lemma}
\begin{proof}
    We use Equation \eqref{eq:CGj1plusj2}, from which the result follows by inspection. We know that
    \begin{align*}
    &\sqrt{j-m_1} \sqrt{j+m_2} C^{j,j,2j}_{m_1,m_2, m_1 + m_2}\\
    &=\sqrt{j-m_1} \sqrt{j+m_2} \sqrt{\frac{(2j)!(2j)!(2j+m_1+m_2)!(2j-m_1-m_2)!}{(4j)!(j+m_1)!(j-m_1)!(j+m_2)!(j-m_2)!}} \\
    &=\sqrt{\frac{(2j)!(2j)!(2j+m_1+m_2)!(2j-m_1-m_2)!}{(4j)!(j+m_1)!(j-m_1-1)!(j+m_2-1)!(j-m_2)!}}.
    \end{align*}
    Similarly,
    \begin{align*}
        &\sqrt{j+m_1 + 1}\sqrt{j-m_2 + 1} C^{j,j,2j}_{m_1+1,m_2-1, m_1+m_2} \\
        &= \sqrt{j+m_1 + 1}\sqrt{j-m_2 + 1}\sqrt{\frac{(2j)!(2j)!(2j+m_1+m_2)!(2j-m_1-m_2)!}{(4j)!(j+m_1+1)!(j-m_1-1)!(j+m_2-1)!(j-m_2+1)!}} \\
        &= \sqrt{\frac{(2j)!(2j)!(2j+m_1+m_2)!(2j-m_1-m_2)!}{(4j)!(j+m_1)!(j-m_1-1)!(j+m_2-1)!(j-m_2)!}}
    \end{align*}
    proving the claim.
\end{proof}
This implies the following lemma.
\begin{lemma}\label{lem:2Jplus1vanish}
    For $\jrel{m}{2J+1}$, and $\jrel{m'}{2J+1}$, it holds that $\spintrunc{2J+1,m,m'} = 0$ in $S_J$.
\end{lemma}
\begin{proof}
    We verify that $\langle f^{2J+1}_{m,m'} \psi^{j_1}_{m_1,m_1'}, \psi^{j_2}_{m_2, m_2'} \rangle = 0$ for all $j_1 m_1 m_1', \; j_2 m_2 m_2' \in \II_S(J)$. Since the degrees of $\psi^{j_1}_{m_1,m_1'}$ and $\psi^{j_2}_{m_2, m_2'}$ are at most $J + \frac{1}{2}$, achieved for $j_1 = j_2 = -J$, and since these degrees must add up to at least $2J + 1$, we are left with $j_1 = j_2 = -J$ as the only possibility for this inner product to be nonzero. Calculating the inner product gives
    \begin{align*}
        &\langle f^{2J+1}_{m,m'} \psi^{-J}_{m_1, m_1'}, \psi^{-J}_{m_2,m_2'} \rangle \\
        &= \frac{1}{2J+2}\left \langle f^{2J+1}_{m,m'} \begin{pmatrix}
        \sqrt{J+1+m_1'} f^{J^+}_{m_1^-, m_1'^+} \\
        \sqrt{J+1-m_1'} f^{J^+}_{m_1^-, m_1'^-}
        \end{pmatrix}, \begin{pmatrix}
        \sqrt{J+1+m_2'} f^{J^+}_{m_2^-, m_2'^+} \\
        \sqrt{J+1-m_2'} f^{J^+}_{m_2^-, m_2'^-}
        \end{pmatrix}\right \rangle \\
        &= \frac{\sqrt{4J+3}}{2J+1} C^{2J+1, J^+, J^+}_{m, m_1^-, m_2^-} \bigg (\sqrt{J+1+m_1'} \sqrt{J+1+m_2'} C^{2J+1, J^+, J^+}_{m', m_1'^+, m_2'^+}\\
        & \phantom{\frac{\sqrt{4J+3}}{2J+1} C^{2J+1, J^+, J^+}_{m, m_1^-, m_2^-} =}
        + \sqrt{J+1 - m_1'} \sqrt{J+1 - m_2'} C^{2J + 1, J^+, J^+}_{m', m_1'^-, m_2'^-} \bigg )
    \end{align*}
    where we factored out the common CGC. To show that this is zero, it suffices to show that the term inside the brackets always vanishes.
    For this, we use the symmetry of CGCs given in Equation \eqref{eq:CGsym}, from which we find that
    \begin{align*}
        &\bigg (\sqrt{J+1+m_1'} \sqrt{J+1+m_2'} C^{2J+1, J^+, J^+}_{m', m_1'^+, m_2'^+}
        + \sqrt{J+1 - m_1'} \sqrt{J+1 - m_2'} C^{2J + 1, J^+, J^+}_{m', m_1'^-, m_2'^-} \bigg ) \\
        &= K \bigg ( \sqrt{J+1+m_1'} \sqrt{J+1+m_2'} C^{J^+, J^+, 2J +1}_{-m_1'^+, m_2'^+, m'} - \sqrt{J+1 - m_1'} \sqrt{J+1 - m_2'}  C^{J^+, J^+, 2J + 1}_{-m_1'^-, m_2'^-, m'}\bigg)
    \end{align*}
    where $K = (-1)^{J + m_1' + 1} \sqrt{\frac{2J+2}{4J+3}}$. The minus sign in the equation appears since the first term gets a $(-1)^{J+m_1' + 1}$ factor, while the second term gets a $(-1)^{J+m_1'}$ factor instead. The total is certainly always zero in the case $m' \neq m_2' - m_1'$, since then the individual CGCs vanish. In the case that $m' = m_2' - m_1 '$ we can apply Lemma \ref{lem:CGDiffzero} to conclude that the term between brackets is zero as well, proving the claim.
\end{proof}
Now we are ready to prove the main theorem.
\begin{proof}[Proof of Theorem \ref{thm:SJdim}]
    First, notice that for $j > 2J + \frac{1}{2}$ and $\jrel{m}{j}, \jrel{m'}{j}$, it holds that $\spintrunc{jmm'} = 0$. Indeed, the case $j = 2J + 1$ is given by Lemma \ref{lem:2Jplus1vanish}. For $j > 2J + 1$, we can count degrees. Since the maximal degree of a spinor $\psi^{j_1}_{m_1, m_1'}$ with $j_1 m_1 m_1' \in \II_S(J)$ is $J + \frac{1}{2}$, and since an inner product $\langle f^{j}_{m,m'} \psi^{j_1}_{m_1, m_1'}, \psi^{j_2}_{m_2, m_2'} \rangle$ with $j > \deg(j_1) + \deg(j_2)$ must vanish, we conclude that $\spintrunc{jmm'} = 0$ for $j > 2J + 1$ as well. \\
    \\  
    We aim to prove that all matrices $\spintrunc{jmm'}$ are linearly independent for $jmm' \in \II(2J+\frac{1}{2})$. That is, for $\{\lambda_{jmm'}\}_{jmm' \in \II(2J+\frac{1}{2})} \subset \C$ with
    \begin{equation*}
        \sum_{jmm \in \II(2J+\frac{1}{2})} \lambda_{jmm'} \spintrunc{jmm'} = 0
    \end{equation*}
    we have that $\lambda_{jmm'} = 0$ for all $jmm' \in \II(2J + \frac{1}{2})$. At coordinates $j_1 m_1 m_1', j_2 m_2 m_2' \in \II_S(J)$ this gives
    \begin{equation*}
        \sum_{jmm \in \II(2J+\frac{1}{2})} \lambda_{jmm'} \langle f^{j}_{m,m'} \psi^{j_1}_{m_1, m_1'}, \psi^{j_2}_{m_2, m_2'} \rangle = 0        
    \end{equation*}
    similarly to the proof of Theorem \ref{thm:RJdim}.\\
    \\
    First we show that $\lambda_{jmm'} = 0$ for $jmm' \in \II(J)$. Fix $j_3 m_3 m_3' \in \II(J)$ and consider the coordinates $j_1 m_1 m_1' = +000$ and $j_2 m_2 m_2' = +j_3 m_3 m_3'$. Then by Lemma \ref{lem:maxspinorsinprod} we find
    \begin{align*}
        0 &=\sum_{jmm' \in \II(2J+\frac{1}{2})} \lambda_{jmm'} \langle f^{j}_{m,m'} \psi^{+0}_{0, 0}, \psi^{+j_3}_{m_3, m_3'} \rangle \\
        &= \sum_{jmm' \in \II(2J+\frac{1}{2})} \lambda_{jmm'}\frac{1}{\sqrt{2j_3+1}} \left \langle f^{j}_{m,m'} \begin{pmatrix}
        0 \\ 1
        \end{pmatrix}, \begin{pmatrix}
        -\sqrt{j_3-m_3'} f^{j_3}_{m_3, m_3' + 1} \\ \sqrt{j_3+1+m_3'} f^{j_3}_{m_3, m_3'}
        \end{pmatrix} \right \rangle \\
        &= \lambda_{j_3m_3m_3'} \frac{\sqrt{j_3+1+m_3'}}{\sqrt{2j_3+1}} 
    \end{align*}
    using orthonormality of matrix coefficients. Since $\jrel{m_3'}{j_3}$ it holds that $\frac{\sqrt{j_3+1+m_3'}}{\sqrt{2j_3+1}} \neq 0$ so that necessarily $\lambda_{j_3 m_3 m_3'} = 0$.\\
    \\
    Now we prove that all other $\lambda_{jmm'}$ vanish as well by an induction argument. Let $L \in \frac{1}{2}\N$ with $J \leq L \leq 2J$. Suppose that we have proven that $\lambda_{jmm'} = 0$ for all $jmm' \in \II(L)$. Then we will show that in fact $\lambda_{jmm'} = 0$ for all $jmm' \in \II(L+\frac{1}{2})$, i.e. $\lambda_{L + \frac{1}{2}, m, m'} = 0$ for all $\jrel{m}{L+\frac{1}{2}}$ and $\jrel{m'}{L+\frac{1}{2}}$. To this end, fix some $m_3, m_3'$ with $\jrel{m_3}{L+\frac{1}{2}}$ and $\jrel{m_3'}{L+\frac{1}{2}}$. We will distinguish the cases $m_3' \geq 0$ and $m_3' < 0$.\\
    \\
    Suppose $m_3' \geq 0$. Define $r = -\sign(m_3) (L-J)$ and consider the coordinates $j_1  = +(L-J), \; m_1 = r, \; m_1' = -(L-J) - 1$ and $j_2 = -J, \; m_2 = m_3 + r + \frac{1}{2},\;  m_2' = m_3' - (L-J) - \frac{1}{2}$. We first show that these are indeed valid coordinates in $\II_S(J)$. \\
    \\
    For $j_1m_1m_1'$, we know that $J \leq L \leq 2J$ so $0 \leq L-J \leq J$ as well, and certainly $\jrelplus{r,-(L-J)-1}{L-J}$, hence $j_1 m_1 m_1' \in \II_S(J)$. For $j_2 m_2 m_2'$, notice that $\jrel{m_3 + r}{J+\frac{1}{2}}$ since $\jrel{m_3}{L+\frac{1}{2}}$, $\jrel{r}{L-J}$, and $m_3$ has the opposite sign of $r$. Then for $m_2 = m_3 + r + \frac{1}{2}$ it holds that $-J \leq m_2 \leq J + 1$. A similar argument works for $m_2'$ using $m_3' \geq 0$, so that we indeed find that $j_2 m_2 m_2'$ is an appropriate coordinate in $\II_S(J)$.\\
    \\
    With this, we can compute 
    \begin{align*}
        \langle f^{j}_{m,m'} \psi^{j_1}_{m_1, m_1'}, \psi^{j_2}_{m_2, m_2'} \rangle &= K\left \langle f^{j}_{m,m'} \begin{pmatrix}
            f^{L-J}_{r, -(L-J)} \\ 0
        \end{pmatrix}, \begin{pmatrix}
        \sqrt{J+\frac{1}{2} + m_3' - (L-J)} f^{J+\frac{1}{2}}_{m_3 + r, m_3' - (L-J)} \\
        \sqrt{J+\frac{3}{2} - m_3' + (L-J)} f^{J+\frac{1}{2}}_{m_3 + r, m_3' - (L-J)-1}
        \end{pmatrix}\right \rangle \\
        &= K' C^{j,L-J,J+\frac{1}{2}}_{m,r,m_3+r} C^{j,L-J,J+\frac{1}{2}}_{m', -(L-J), m_3' - (L-J)}
    \end{align*}
    where $K$ and $K'$ are nonzero constants, explicitly
    $K' = -\frac{\sqrt{2j+1}}{2J+2} \sqrt{2J + \frac{1}{2} + m_3' - L}$.
    For the CGCs to be nonvanishing, we certainly need $m = m_3$ and $m' = m_3'$ by the addition of weights condition. Furthermore, we require $J + \frac{1}{2} \in \{|j-(L-J)|, \ldots, j + L - J\}$, which implies that $j \leq L + \frac{1}{2}$ since otherwise $j-L+J > J + \frac{1}{2}$. For $j = L + \frac{1}{2}$ we know that $C^{L+\frac{1}{2},L-J,J+\frac{1}{2}}_{m_3,r,m_3+r} C^{L+\frac{1}{2},L-J,J+\frac{1}{2}}_{m_3', -(L-J), m_3' - (L-J)} \neq 0$ by Lemma \ref{lem:CGnonvanish}, so also $\spintrunc{L+\frac{1}{2}, m_3, m_3'}^{j_1 m_1 m_1'}_{j_2 m_2 m_2'} \neq 0$. Thus, for this choice of $j_1 m_1 m_1', j_2 m_2 m_2'$ it holds that
    \begin{equation*}
            \sum_{jmm \in \II(2J+\frac{1}{2})} \lambda_{jmm'} \spintrunc{jmm'}^{j_1 m_1 m_1'}_{j_2 m_2 m_2'} = \sum_{j = \sigma}^{L+\frac{1}{2}} \lambda_{jm_3m_3'} \spintrunc{jm_3m_3'}^{j_1 m_1 m_1'}_{j_2 m_2 m_2'} = 0
    \end{equation*}
    where the summation takes integer steps starting at $\sigma = 0$ or $\sigma = \frac{1}{2}$ depending on the parity of $L + \frac{1}{2}$. However, by hypothesis $\lambda_{jm_3 m_3'} = 0$ for all $j \leq L$, so this reduces to $$\lambda_{L+\frac{1}{2}, m_3, m_3'} \spintrunc{L+\frac{1}{2}, m_3,m_3'}^{j_1 m_1 m_1'}_{j_2 m_2 m_2'} = 0.$$ But then necessarily $\lambda_{L+\frac{1}{2}, m_3, m_3'} = 0$ as desired.\\
    \\
    Now suppose that $m_3' < 0$. With $r = -\sign(m_3)(L-J)$ we take $j_1 = L - J$, $m_1 = r$ and $m_1' = L-J$ instead, together with $j_2 = -J$, $m_2 = m_3 + r + \frac{1}{2}$ and $m_2' = m_3' + L-J + \frac{1}{2}$. Like before, these are valid coordinates in $\II_S(J)$, using the same kind of reasoning as in the case $m_3 \geq 0$. Then we can compute for arbitrary $jmm'$ that
    \begin{align*}
        \langle f^{j}_{m,m'} \psi^{j_1}_{m_1, m_1'}, \psi^{j_2}_{m_2, m_2'} \rangle &= K\left \langle f^{j}_{m,m'} \begin{pmatrix}
            0 \\ f^{L-J}_{r, L-J}
        \end{pmatrix}, \begin{pmatrix}
        \sqrt{ m_3' + L + \frac{3}{2}} f^{J+\frac{1}{2}}_{m_3 +r, m_3' + L-J+1} \\
        \sqrt{2J - m_3 - L + \frac{1}{2}} f^{J+\frac{1}{2}}_{m_3 +r, m_3' + L-J}
        \end{pmatrix}\right \rangle  \\
        &= K' C^{j,L-J,J+\frac{1}{2}}_{m,r,m_3+r} C^{j,L-J,J+\frac{1}{2}}_{m',L-J,m_3' + L - J}
    \end{align*}
    where $K$ and $K'$ are nonzero constants. Analogously to the case $m_3' \geq 0$, we need that $m = m_3$, $m' = m_3'$ and $j \leq L + \frac{1}{2}$ for the CGCs to be nonvanishing, and for $j = L + \frac{1}{2}$ it holds that $ C^{j,L-J,J+\frac{1}{2}}_{m_3,r,m_3+r} C^{j,L-J,J+\frac{1}{2}}_{m_3',L-J,m_3' + L - J} \neq 0$ by Lemma \ref{lem:CGnonvanish}, hence $\spintrunc{L+\frac{1}{2}, m_3, m_3'}^{j_1 m_1 m_1'}_{j_2 m_2 m_2'} \neq 0$. Thus, for this choice of $j_1 m_1 m_1', j_2 m_2 m_2'$ it holds that
    \begin{equation*}
            \sum_{jmm \in \II(2J+\frac{1}{2})} \lambda_{jmm'} \spintrunc{jmm'}^{j_1 m_1 m_1'}_{j_2 m_2 m_2'} = \sum_{j = \sigma}^{L+\frac{1}{2}} \lambda_{jm_3m_3'} \spintrunc{jm_3m_3'}^{j_1 m_1 m_1'}_{j_2 m_2 m_2'} = 0
    \end{equation*}
    as before. Using the hypothesis $\lambda_{jm_3 m_3'} = 0$ for all $j \leq L$, we find $\lambda_{L+\frac{1}{2}, m_3, m_3'} = 0$.
    \end{proof}

\setcounter{subsection}{0}
\subsection{$C^*$-envelope and propagation number}
\setcounter{subsection}{8}
As with the Peter-Weyl truncation $R_J$, we are interested in the propagation number of $S_J$ as an operator system invariant. It turns out that the propagation number of $S_J$ is actually strictly larger than that of $R_J$ for small values of $J$. This should be contrasted with the case of the $d$-dimensional torus $\mathbb{T}^d$, studied in \cite{Leimbach24}. There, the propagation number of the spectrally truncated operator system is equal to $2$, which is the same for the Peter-Weyl truncation. The underlying reason for this is that the Dirac operator on the torus respects the matrix coefficients of irreducible representations, while the Dirac operator on $\SU(2)$ does not. More precisely, on the torus, the matrix coefficients are $e^{in \cdot x} \in C(\mathbb{T}^d)$ for $n \in \Z^d$, $x \in \mathbb{T}^d$, and the spectral truncation at level $\Lambda$ gives 
\begin{equation*}
    P_\Lambda L^2(S_{\mathbb{T}^d}) = \bigoplus_{|n| \leq \Lambda} W_n \tensor \C^{2^{\lfloor \frac{d}{2} \rfloor}} \cong \left (\bigoplus_{|n| \leq \Lambda} W_n \right ) \tensor \C^{2^{\lfloor \frac{d}{2} \rfloor}}
\end{equation*}
where $W_n = \spann \{e^{in \cdot x}\}$ for an $n \in \Z^d$. However, the part within brackets on the right-hand side is precisely the Peter-Weyl truncation. In the case of $\SU(2)$ the spectral truncation at level $J$ gives
\begin{equation*}
    P_J L^2(S_{\SU(2)}) =\bigoplus_{j \leq J} \left \{ V_{j} \tensor V_{j+\frac{1}{2}} \tensor \Hom_G(V_{j} \tensor V_{j+\frac{1}{2}}, V_{\frac{1}{2}}) \right \} \oplus \left \{ V_{j+\frac{1}{2}} \tensor V_{j} \tensor \Hom_G(V_{j+\frac{1}{2}} \tensor V_{j}, V_{\frac{1}{2}}) \right \}
\end{equation*}
which is not rewritable in terms of the Peter-Weyl truncation of $\SU(2)$. As such, the propagation numbers of $S_J$ and $R_J$ are more difficult to relate.\\
\\
The proof of Theorem \ref{thm:RJpropnr} in terms of kernels does not easily carry over to the vector-valued situation, as the kernels are then matrix-valued. Instead, the proof presented in this section works by writing all matrix units as a linear combination of products of matrices. For this, we fix the following notation used for the remainder of the section.

\begin{definition}
    Let $j_1 m_1 m_1', \; j_2 m_2 m_2' \in \II_S(J)$, where we recall that this means the sign of $j_i$ can be either $+$ or $-$ for $i = 1,2$. With $e^{j_1, m_1, m_1'}_{j_2, m_2, m_2'}$ we denote the matrix satisfying $e^{j_1, m_1, m_1'}_{j_2, m_2, m_2'} \psi^{j_1}_{m_1, m_1'} = \psi^{j_2}_{m_2, m_2'}$ while for $jmm \neq j_1 m_1 m_1'$ it satisfies $e^{j_1, m_1, m_1'}_{j_2, m_2, m_2'} \psi^j_{m,m'} = 0$.
\end{definition}
The main theorem in this section is the following.
\begin{theorem}\label{thm:SJpropnr}
    The $C^*$-envelope of $S_J$ is $M_N(\C)$, where $N = \dim P_J L^2(S)= \frac{2}{3}(2J+1)(2J+2)(2J+3)$. Furthermore, for all $J \in \frac{1}{2}\N$, the propagation number satisfies $2 \leq \prop(S_J) \leq 8$.
\end{theorem}
Unlike Theorem \ref{thm:RJpropnr}, the case $J = 0$ is nontrivial, and has a propagation number different from $1$. Since the proof is quite lengthy and contains several lemmas, we first outline the plan of the proof.\\
\\
The hard part of the proof is the upper bound of 8. This is also how we determine that $C^*_{env}(S_J) = M_N(\C)$. Once we have established that this is the $C^*$-envelope, the lower bound of 2 follows from a dimension count. Thus, we focus on determining the upper bound. The main difference between matrices in $S_J$ compared to $R_J$, is that matrices in $S_J$ typically have groups of four nonzero entries. To be more precise, if $\spintrunc{jmm'} \in S_J$ has $\spintrunc{jmm'}^{+j_1 m_1 m_1'}_{+j_2 m_2 m_2'} \neq 0$, then the triple $-j_1^{-}, m_1^+, m_1'^+$ instead of $j_1 m_1 m_1'$ can also yield a nonzero entry, since $\psi^{-j_1^-}_{m_1^+, m_1^+}$ contains the same matrix coefficients as $\psi^{+j_1}_{m_1, m_1'}$, namely $f^{j_1}_{m_1, m_1'}$ and $f^{j_1}_{m_1, m_1'+1}$. This also happens for $-j_2^{-}, m_2^+, m_2'^+$ instead of $j_2 m_2 m_2'$, yielding a total of four cases corresponding to the same $jmm', \; j_1m_1m_1', \; j_2m_2 m_2'$.\\
\\
Globally, the proof of the theorem consists of the following steps:
\begin{itemize}
    \item[Step 1:] To show that $\prop(S_J) \leq 8$, we may equivalently show that for every pair of coordinates $j_1 m_1 m_1', \; j_2 m_2 m_2' \in \II_S(J)$ it holds that $e^{j_1, m_1, m_1'}_{j_2, m_2, m_2'} \in S_J^{\circ 8}$. For this it is enough to show that $e^{j_1, m_1, m_1'}_{-J,-J,-J} \in S_J^{\circ 4}$. 
    \item[Step 2:] Identify matrices in $S_J^{\circ 3}$ with only one nonzero entry. In Lemma \ref{lem:propnrsinglemat} we prove that $$\spintrunc{\tfrac{1}{2}, -\tfrac{1}{2}, \tfrac{1}{2} }
    \spintrunc{2J+ \tfrac{1}{2}, 2J + \tfrac{1}{2}, 2J + \tfrac{1}{2}} \in S_J^{\circ 2}$$ equals $e^{+J, -J, -J-1}_{+J, J, J}$ up to a constant. From this we can multiply with certain matrices to find $e^{+J,J,J}_{-J,-J,-J} \in S^{\circ 3}$ and $e^{+J, -J, -J-1}_{-J, -J, -J} \in S^{\circ 3}$.
    \item[Step 3:] Identify matrices in $S_J$ with nonzero $e^{j_1, m_1, m_1'}_{J, J, J}$ and $e^{j_1, m_1, m_1'}_{J, -J, -J-1}$ contributions. For positive $j_1$, these are $\spintrunc{J+j_1, J-m_1, J-m_1'}$ and $\spintrunc{J+j_1, -J-m_1, -J-m_1' - 1}$ respectively. These matrices also come with $e^{-j_1^{-}, m_1^+, m_1'^+}_{+J,J,J}$ and $e^{-j_1^{-}, m_1^+, m_1'^+}_{+J,-J,-J-1}$ terms each, and more terms corresponding to higher $j$ values.
    \item[Step 4:] Multiply with $e^{+J,J,J}_{-J,-J,-J}$ and $e^{+J,-J,-J-1}_{-J,-J,-J}$. This way, we eliminate unwanted terms, and furthermore convert $e^{j_1,m_1,m_1'}_{J,-J,-J-1}$ and $e^{j_1,m_1,m_1'}_{J, J, J}$ terms to $e^{j_1,m_1,m_1'}_{-J,-J,-J}$ terms, allowing us to cancel them later in linear combinations. These products are in $S_J^{\circ 4}$.
    \item[Step 5:] Use linear independence to separate single terms. From step 4 we have two matrices with both $e^{j_1, m_1, m_1'}_{-J,-J,-J}$ and $e^{-j_1^-, m_1^+, m_1'^+}_{-J,-J,-J}$ terms. The matrices that we use turn out to be linearly independent in a way that allows us to find matrices in $S_J^{\circ 4}$ with an $e^{j_1, m_1, m_1'}_{-J,-J,-J}$ term, but no $e^{-j_1^-, m_1^+, m_1'^+}_{-J,-J,-J}$ term, and vice versa. This is shown in Lemma \ref{lem:Mlempropnr}.
    \item[Step 6:] Get rid of $e^{j, m_1, m_1'}_{-J,-J,-J}$ terms for higher $j$. To accomplish this, we use a recursive structure, eliminating terms from `above' starting at $2J$ or $2J+\frac{1}{2}$ depending on parity. This isolates $e^{j_1, m_1, m_1'}_{-J, -J, -J} \in S_J^{\circ 4}$, proving the claim.
\end{itemize}

Now we start with the precise statements. First, we find matrices with exactly one nonzero entry.
\begin{lemma}\label{lem:propnrsinglemat}
    Let $J \in \frac{1}{2}\N$. Then
    \begin{equation}\label{eq:lempropnr1}
        \spintrunc{2J+\tfrac{1}{2}, -2J - \tfrac{1}{2}, -2J - \tfrac{1}{2}}
        \spintrunc{\tfrac{1}{2}, -\tfrac{1}{2}, \tfrac{1}{2}} \spintrunc{2J + \tfrac{1}{2}, 2J + \tfrac{1}{2}, 2J + \tfrac{1}{2}} = L_1 e^{+J, -J, -J-1}_{-J,-J,-J}
    \end{equation}
    and
    \begin{equation}\label{eq:lempropnr2}
    \spintrunc{\tfrac{1}{2},-\tfrac{1}{2}, \tfrac{1}{2}} \spintrunc{2J+\tfrac{1}{2}, -2J - \tfrac{1}{2}, -2J - \tfrac{1}{2}} \spintrunc{\tfrac{1}{2}, \tfrac{1}{2}, -\tfrac{1}{2}}= L_2 e^{+J, J, J}_{-J,-J,-J}
    \end{equation}
    where $L_1, L_2$ are nonzero constants. In particular, $e^{+J, -J, -J-1}_{-J,-J,-J} \in S_J^{\circ 3}$ and $e^{+J,J,J}_{-J,-J,-J} \in S_J^{\circ 3}$.
\end{lemma}
\begin{proof}
    This follows from several straightforward computations. In the proof we will use the equality $\spintrunc{j,m,m'}^* = (-1)^{m-m'}\spintrunc{j,-m,-m'}$, which can be deduced from $\overline{f^{j}_{m,m'}} = (-1)^{m-m'} f^{j}_{-m,-m'}$ as shown in (\cite{Edmonds1996-xa}, Eq. 4.2.7).\\
    \\
    First, we show that
    \begin{equation*}
        \spintrunc{2J+ \tfrac{1}{2}, 2J + \tfrac{1}{2}, 2J + \tfrac{1}{2}} = K_1 e^{J,-J,-J-1}_{-J, J+1, J} + K_2 e^{-J, -J, -J}_{J, J, J}
    \end{equation*}
    where $K_1, K_2$ are nonzero constants. The coefficient on column $j_1 m_1 m_1'$ and row $j_2 m_2 m_2'$ is given by
    \begin{equation}\label{eq:2Jplushalfinprod}
    \spintrunc{2J+\tfrac{1}{2}, 2J + \tfrac{1}{2}, 2J + \tfrac{1}{2}}^{j_1 m_1 m_1'}_{j_2 m_2 m_2'} = \langle f^{2J + \tfrac{1}{2}}_{2J + \tfrac{1}{2}, 2J + \tfrac{1}{2}} \psi^{j_1}_{m_1, m_1'}, \psi^{j_2}_{m_2, m_2'} \rangle.
    \end{equation}
    In order for this inner product to be nonzero, we certainly need $$\deg(j_2) \in \{2J + \tfrac{1}{2} - \deg(j_1), 2J + \tfrac{3}{2} - \deg(j_1), \ldots, 2J + \tfrac{1}{2} + \deg(j_1)\}.$$
    Since $\deg(j_1), \deg(j_2) \leq J + \frac{1}{2}$, we get two distinct possibilities: Either $\deg(j_1) = J + \frac{1}{2}$ and $\deg(j_2) = J$, or $\deg(j_1) = J$ and $\deg(j_2) = J + \frac{1}{2}$. In the case $\deg(j_1) = J + \frac{1}{2}$, we find the options $j_1 = -J, j_2 = -J^-$ and $j_1 = -J, j_2 = J$. However, $j_2 = -J^{-}$ will not give a contribution. Indeed, the inner product in Equation \eqref{eq:2Jplushalfinprod} in this case is given by
    \begin{equation*}
    \langle f^{2J + \frac{1}{2}}_{2J + \frac{1}{2}, 2J + \frac{1}{2}} \psi^{-J}_{m_1, m_1'}, \psi^{-J^-}_{m_2, m_2'} \rangle,
    \end{equation*}
    and for this to be nonzero we require $2J + \frac{1}{2} = m_2' - m_1'$. However, since $\jrel{m_2'}{J^-}$ in this case, $m_2' \leq J - \frac{1}{2}$, while $m_1' \geq -J$ similarly. But then $m_2' - m_1' \leq 2J - \frac{1}{2} < 2J + \frac{1}{2}$, so the weight condition is never met. Thus, we may continue to the $j_2 = J$ possibility. In this case, the inner product in Equation \eqref{eq:2Jplushalfinprod} becomes
    \begin{equation*}
    \langle f^{2J + \frac{1}{2}}_{2J + \frac{1}{2}, 2J + \frac{1}{2}} \psi^{-J}_{m_1, m_1'}, \psi^{J}_{m_2, m_2'} \rangle
    \end{equation*}
    instead. The weight conditions in this case are $2J + \frac{1}{2} + m_1 - \frac{1}{2} = m_2$ and $2J + \frac{1}{2} + m_1' - \frac{1}{2} = m_2'$. Considering the constraints on $m_1, m_1', m_2$ and $m_2'$, we find that $m_1 = -J$, $m_1' = -J$, $m_2 = J$ and $m_2' = J$ is the only possibility for the magnetic moments. To calculate the inner product, we use Lemma \ref{lem:maxspinorsinprod} which gives
    \begin{equation*}
        \langle f^{2J + \frac{1}{2}}_{2J + \frac{1}{2}, 2J + \frac{1}{2}} \psi^{-J}_{-J, -J}, \psi^{J}_{J,J} \rangle = \sqrt{4J+2} \; C^{2J+\frac{1}{2}, J + \frac{1}{2}, J}_{2J+\frac{1}{2}, -J-\frac{1}{2}, J} C^{2J+\frac{1}{2}, J + \frac{1}{2}, J}_{2J + \frac{1}{2}, -J-\frac{1}{2}, J}
    \end{equation*}
    for the inner product. The CGCs in this expression are nonvanishing by Lemma \ref{lem:CGnonvanish}, so that we conclude that $\spintrunc{2J+\frac{1}{2},2J+\frac{1}{2},2J+\frac{1}{2}}^{-J, -J, -J}_{J,J,J}$ is nonvanishing.\\
    \\
    The case $\deg(j_1) = J$ and $\deg(j_2) = J + \frac{1}{2}$ is similar. Here we find that $j_2 = -J$, and $j_1$ could either be $-J^-$ or $+J$. However, when $j_1 = -J^{-}$, the magnetic moment condition $2J + \frac{1}{2} + m_1' = m_2'$ cannot be fulfilled for the same reason as in the earlier case. When $j_1 = +J$, the magnetic moments can only be fulfilled with $m_1 = -J, m_1' = -J-1$ and $m_2 = J + 1, m_2' = J$. Then the CGCs arising from the inner product are again nonzero by Lemma \ref{lem:CGnonvanish}. Thus, we find that
    \begin{equation}\label{eq:2Jplushalfmatrix}
        \spintrunc{2J+ \tfrac{1}{2}, 2J + \tfrac{1}{2}, 2J + \tfrac{1}{2}} = K_1 e^{J,-J,-J-1}_{-J, J+1, J} + K_2 e^{-J, -J, -J}_{J, J, J}
    \end{equation}
    as desired.\\
    \\
    Now we take a product, and show that
    \begin{equation}\label{eq:propnr2prod}
        \spintrunc{\tfrac{1}{2}, -\tfrac{1}{2}, \tfrac{1}{2} }
        \spintrunc{2J+ \tfrac{1}{2}, 2J + \tfrac{1}{2}, 2J + \tfrac{1}{2}} = K e^{+J, -J, -J-1}_{+J, J, J}
    \end{equation}
    where $K$ is some nonzero constant.
    Since we have shown that the column indices $jmm'$ for which $\spintrunc{2J + \frac{1}{2}, 2J + \frac{1}{2}, 2J + \frac{1}{2}}^{\cdot, \cdot, \cdot}_{jmm'}$ is nonzero are precisely $-J, J+1, J$ and $J,J,J$, 
    it suffices to determine how $\spintrunc{\frac{1}{2}, -\frac{1}{2}, \frac{1}{2}}$ acts on those coordinates. First, we consider $-J, J+1, J$, where contributions by $\spintrunc{\tfrac{1}{2},-\tfrac{1}{2},\tfrac{1}{2}}$ consist of
    \begin{equation}
        \spintrunc{\tfrac{1}{2},-\tfrac{1}{2},\tfrac{1}{2}}^{-J, J+1, J}_{j, m, m'} = \langle f^{\frac{1}{2}}_{-\frac{1}{2}, \frac{1}{2}} \psi^{-J}_{J+1, J}, \psi^{j}_{m,m'} \rangle
    \end{equation}
    for $jmm' \in \II_S(J)$. Taking into account the total degrees in the inner product, we find $\deg(j) \in \{J, J+1\}$, which combined with $\deg(j) \leq J + \frac{1}{2}$ gives $j = +J$ or $j = -J^-$. For $j = +J$, the weights are forced to be $m = m' = J$, in which case we get
    \begin{equation*}
        \langle f^{\frac{1}{2}}_{-\frac{1}{2}, \frac{1}{2}} \psi^{-J}_{J+1, J}, \psi^{+J}_{J,J} \rangle = \sqrt{\frac{2}{2J+1}} C^{\frac{1}{2}, J+\frac{1}{2}, J}_{-\frac{1}{2}, J+\frac{1}{2}, J} C^{\frac{1}{2}, J + \frac{1}{2}, J}_{\frac{1}{2}, J - \frac{1}{2}, J}
    \end{equation*}
    using Lemma \ref{lem:maxspinorsinprod}, and thus the inner product is nonvanishing by Lemma \ref{lem:CGnonvanish}. For $j = -J^{-}$, we instead get the weight requirement $m' = J + \frac{1}{2}$, which is in contradiction with $\jrel{m'}{J - \frac{1}{2}}$. Thus, there is no term corresponding to $j = -J^{-}$.\\
    \\
    For the $J,J,J$ coordinates, it turns out that we find zero. The relevant coefficients of $\spintrunc{\tfrac{1}{2},-\tfrac{1}{2},\tfrac{1}{2}}^{J, J, J}_{j, m, m'}$ are
    \begin{equation}
        \spintrunc{\tfrac{1}{2},-\tfrac{1}{2},\tfrac{1}{2}}^{J, J, J}_{j, m, m'} = \langle f^{\frac{1}{2}}_{-\frac{1}{2}, \frac{1}{2}} \psi^{J}_{J, J}, \psi^{j}_{m,m'} \rangle
    \end{equation}
    for $jmm' \in \II_S(J)$. Here we have the requirement $\deg(j) \in \{J-\frac{1}{2}, J + \frac{1}{2}\}$, which leads to the three possibilities $j = -J$, $j = -(J-1)$ and $j = J-\frac{1}{2}$. In the first and second of these cases, we find $m' - \frac{1}{2} = \frac{1}{2} + J$, so $m' = J + 1$, but this is impossible with these values of $j$. For the last case $j = J - \frac{1}{2}$, we instead find $m' = \frac{1}{2} + J$, but this also does not satisfy $\jrel{m'}{J-\frac{1}{2}}$. Thus, by multiplying the matrices $\spintrunc{\frac{1}{2}, -\frac{1}{2}, \frac{1}{2} }$ and 
    $\spintrunc{2J+ \frac{1}{2}, 2J + \frac{1}{2}, 2J + \frac{1}{2}}$ using Equation \eqref{eq:2Jplushalfmatrix} we obtain Equation \eqref{eq:propnr2prod}.\\
    \\
    From these computations, it readily follows that Equation \eqref{eq:lempropnr1} holds. Since $$\spintrunc{2J+\tfrac{1}{2}, -2J - \tfrac{1}{2}, -2J - \tfrac{1}{2}} = \spintrunc{2J + \tfrac{1}{2}, 2J + \tfrac{1}{2}, 2J + \tfrac{1}{2}}^* = K_1 e^{-J, J + 1, J}_{J, -J, -J-1} + K_2 e^{J, J, J}_{-J, -J, -J},$$ we can use Equation \eqref{eq:propnr2prod} and contract indices to find Equation \eqref{eq:lempropnr1}.\\
    \\
    It remains to show Equation \eqref{eq:lempropnr2}. First,
    \begin{equation*}
        \spintrunc{2J+\tfrac{1}{2}, -2J - \tfrac{1}{2}, -2J - \tfrac{1}{2}} \spintrunc{\tfrac{1}{2}, \tfrac{1}{2}, - \tfrac{1}{2}} = \left (-\spintrunc{\tfrac{1}{2}, -\tfrac{1}{2}, \tfrac{1}{2} }
        \spintrunc{2J+ \tfrac{1}{2}, 2J + \tfrac{1}{2}, 2J + \tfrac{1}{2}} \right )^* = K e^{J, J, J}_{J, -J, -J-1}
    \end{equation*}
    using Equation \eqref{eq:propnr2prod}. Thus, when taking the product with $\spintrunc{\frac{1}{2}, - \frac{1}{2}, \frac{1}{2}}$ from the left, it suffices to consider how $\spintrunc{\frac{1}{2}, - \frac{1}{2}, \frac{1}{2}}$ acts on $\psi^{J}_{-J, -J-1}$. For an arbitrary $jmm' \in \II_S(J)$, the inner product
    \begin{equation*}
    \langle f^{\frac{1}{2}}_{\frac{1}{2}, -\frac{1}{2}} \psi^{J}_{-J,-J-1}, \psi^{j}_{m,m'} \rangle
    \end{equation*}
    corresponds to the $e^{J, -J, -J-1}_{j,m,m'}$ term of $\spintrunc{\frac{1}{2}, -\frac{1}{2}, \frac{1}{2}}$. Using a similar analysis as before, we find that $j = -J, m = -J, m' = -J$ is the only triple making this inner product nonzero, which yields Equation \eqref{eq:lempropnr2} and proves the theorem.
\end{proof}

In order to separate the $e^{j, m, m'}_{-J,-J,-J}$ and $e^{-j^-, m^+, m'^+}_{-J,-J,-J}$ terms later, we need the following lemma on linear independence of certain vectors.
\begin{lemma}\label{lem:Mlempropnr}
    Fix $j, J \in \frac{1}{2}\N$ with $j \leq J$ and let $\jrel{m}{j}, \; \jrel{m'}{j}$ with $m' \neq j$. Consider the square matrix
    \begin{equation*}
        M = \begin{pmatrix}
        \langle f^{J+j}_{-J-m, -J - m' - 1} \psi^{j}_{m, m'}, \psi^{J}_{-J, -J-1} \rangle & \langle f^{J + j}_{-J -m, -J-m' - 1} \psi^{-j^-}_{m^+, m'^+}, \psi^{J}_{-J, -J-1} \rangle \\
        \langle f^{J+j}_{J-m, J-m'} \psi^{j}_{m, m'}, \psi^{J}_{J,J} \rangle & \langle f^{J+j}_{J-m, J-m'} \psi^{-j^{-}}_{m^+, m'^+}, \psi^J_{J,J} \rangle
        \end{pmatrix}.
    \end{equation*}
    Then $M$ is invertible, i.e. $\det M \neq 0$. In particular, the rows of $M$ are linearly independent.
\end{lemma}
\begin{proof}
    In the proof we show that the determinant is nonzero by showing that the matrix entries are all nonzero, and that the signs of the matrix entries cannot cancel each other out in the determinant computation. \\
    \\
    Let $M = \begin{pmatrix}
    a & b \\ c & d
    \end{pmatrix}$ be the matrix above. We first consider $a = \langle f^{J+j}_{-J-m, -J - m' - 1} \psi^{j}_{m, m'}, \psi^{J}_{-J, -J-1} \rangle$, determine that $a \neq 0$, and determine the sign of $a$. Since $-j \leq m' \leq j-1$ it holds that $\jrel{-m' - 1}{j}$, and thus $\jrel{-J-m'-1}{J+j}$ as well. Hence, $f^{J+j}_{-J-m, -J-m'-1}$ is a nonzero matrix coefficient. By Lemma \ref{lem:maxspinorsinprod} we know
    \begin{equation*}
        a = \sqrt{\frac{2(J+j)+1}{2J+1}}\sqrt{j-m'}C^{J+j, j, J}_{-J-m, m, -J}C^{J+j,j,J}_{-J-m'-1,m'+1,-J}
    \end{equation*}
    where it should be noted that $\sqrt{j-m'} \neq 0$ as we assumed $m' \neq j$. Applying the symmetry described in Equation \eqref{eq:CGsym} to the first CGC gives
    \begin{equation*}
        C^{J+j,j,J}_{-J-m,m,-J} = (-1)^{j+m} \sqrt{\frac{2J+1}{2(J+j)+1}} C^{j, J, J+j}_{-m,-J,-J-m}
    \end{equation*}
    while for the second CGC we find
    \begin{equation*}
    C^{J+j,j,J}_{-J-m'-1,m'+1,-J} = (-1)^{j+m'+1}\sqrt{\frac{2J+1}{2(J+j)+1}} C^{j, J, J+j}_{-m'-1, -J, -J-m'-1}
    \end{equation*}
    instead. By Equation \eqref{eq:CGj1plusj2} we know that the coefficients $C^{j_1, j_2, j_1+j_2}_{m_1,m_2,m_1+m_2}$ are always strictly positive. Thus, also $a$ is nonzero, with sign equal to $(-1)^{2j+m+m'+1}$. A similar computation for the other coefficients shows that the sign of $b, c$ and $d$ are all equal to $(-1)^{2j+m+m'}$. With this we can determine the sign of the determinant, $\det M = ad - bc$. The sign of $ad$ is $(-1)^{4j+2m+2m' + 1}$, while the sign of $-bc$ is $(-1)^{4j+2m+2m' + 1}$ as well. Thus, $ad$ and $-bc$ have the same sign and cannot cancel each other out, which implies that $\det M \neq 0$.
\end{proof}
Finally, we use the previous results to prove the upper bound for the propagation number.
\begin{proof}[Proof of Theorem \ref{thm:SJpropnr}]
    We aim to show that $e^{j_1, m_1, m_1'}_{j_2, m_2, m_2'} \in S_J^{\circ 8}$ for all $j_1 m_1 m_1', \; j_2 m_2 m_2' \in \II_S(J)$. In order to do this, we show that $e^{j_1, m_1, m_1'}_{-J, -J, -J} \in S_J^{\circ 4}$. Then also the adjoint is in $S_J^{\circ 4}$, and we get $e^{j_1, m_1, m_1'}_{j_2, m_2, m_2'} = (e^{j_2, m_2, m_2'}_{-J, -J, -J})^* e^{j_1, m_1, m_1'}_{-J,-J,-J} \in S_J^{\circ 8}$.\\
    \\
    Fix $j_2m_2m_2' \in \II_S(J)$. In the case that $j_2$ is positive, we additionally assume $m_2' \neq j_2$ and $m_2' \neq -j_2 - 1$ for now. These cases are similar, but we can deal with them in a better way later, so we postpone those cases until the end of the proof. Then we want to show that $e^{j_2, m_2, m_2'}_{-J,-J,-J} \in S_J^{\circ 4}$. To treat negative and positive $j_2$ simultaneously, we define an auxiliary triple $j_1 m_1 m_1'$. In the case that $j_2$ is negative, define $j_1, m_1, m_1' = |j_2|^+, m_2^{-}, m_2'^{-}$, and in the case that $j_2$ is positive, consider $j_1 m_1 m_1' = j_2 m_2 m_2'$. Define $\sigma \in \{ 0, \frac{1}{2} \}$ as $\sigma = J + j_1 \mod 1$, so the fractional part of $J + j_1$. Consider the matrices $$A_i := \spintrunc{J+i, J - m_1, J - m_1'}$$ and 
    $$B_i := \spintrunc{J+i, -J-m_1, -J-m_1'- 1}$$
    for $i \in \{j_1, j_1 + 1, \ldots, J+\sigma\}$. The labels in the truncated matrices $A_i$ and $B_i$ are always valid since our assumptions imply $m_1' \neq j_1$ and $m_1' \neq -j_1-1$. We will combine these matrices using $e^{J, J, J}_{-J, -J, -J}\in S_J^{\circ 3}$ and $e^{J, -J, -J-1}_{-J, -J, -J} \in S_J^{\circ 3}$, so we are interested in their $J, J, J$ and $J, -J, -J-1$ components. Firstly, we have that
    \begin{equation}\label{eq:Aicoeff}
        (A_i)^{jmm'}_{J, J, J} = \langle f^{J+i}_{J-m_1, J-m_1'} \psi^{j}_{m,m'}, \psi^{J}_{J,J} \rangle
    \end{equation}
    for arbitrary $jmm'$, which forces $J \in \{J+i-\deg(j), \ldots, J+i + \deg(j)\}$. Combined with the constraint $\deg(j) \leq J + \frac{1}{2}$, this is equivalent to $\deg(j) \in \{i, i+1, \ldots, J+\sigma\}$. For positive $j$, we find that $m = m_1$ and $m' = m_1'$ are forced, while for negative $j$ we need $m = m_1^+$ and $m' = m_1'^+$ instead. This yields the decomposition
  
    \begin{equation*}
        A_i = \sum_{j=i}^{J+\sigma} (A_i)^{j, m_1, m_1'}_{J, J, J} e^{j, m_1, m_1'}_{J, J, J} +  (A_i)^{-j^{-}, m_1^+, m_1'^+}_{J, J, J} e^{-j^{-}, m_1^+, m_1'^+}_{J, J, J} + O
    \end{equation*}
    where $O$ is a matrix with $O^{jmm}_{J, J, J} = 0$ for all $jmm' \in \II_S(J)$. In this equation we use a convention that $e^{+J+\frac{1}{2}, m_1, m_1'}_{a, b, c} = 0$ since $\psi^{J+\frac{1}{2}}_{m_1, m_1'} \in \ker P_J$. Thus, when $\sigma = \frac{1}{2}$ the $i = J+\sigma$ level in the sum only contains one term, corresponding to $-j^-$, whereas the summand otherwise contains two terms. Furthermore, by Lemma \ref{lem:maxspinorsinprod} it holds that $(A_i)^{jmm'}_{J,J,J} \neq 0$ for all the $jmm'$ summed over in the equation, except for the aforementioned $j = +J+\frac{1}{2}$ term. \\
    \\
    For $B_i$, we are instead interested in components of the form
    \begin{equation}\label{eq:Bicoeff}
        (B_i)^{jmm'}_{J,-J,-J-1} = \langle f^{J+i}_{-J-m_1, -J-m_1' - 1} \psi^j_{m,m'}, \psi^J_{-J, -J-1} \rangle
    \end{equation}
    for arbitrary $jmm'$. Similar to $A_i$, we again find $\deg(j) \in \{i, i+1, \ldots, J+\sigma \}$. For positive $j$, $m = m_1, m' = m_1'$ is forced, while for negative $j$ we have $m = m_1^+, m_1' = m_1'^+$ again. Hence, we may decompose

    \begin{equation*}
    B_i = \sum_{j=i}^{J+\sigma} (B_i)^{j,m_1,m_1'}_{J,-J,-J-1} e^{j,m_1,m_1'}_{J,-J,-J-1} + (B_i)^{-j^-,m_1^+, m_1'^+}_{J,-J,-J-1} e^{-j^-,m_1^+, m_1'^+}_{J, -J, -J-1}+ U
    \end{equation*}
    where $U$ has $U^{jmm'}_{J,-J,-J-1}$ for all $jmm' \in \II_S(J)$. Similarly, when $\sigma = \frac{1}{2}$ the $i = J + \sigma$ in the summation only yields one nonzero term, as we use the convention $e^{J + \frac{1}{2}, m_1, m_1'}_{a,b,c} = 0$.\\
    \\
    Now we aim to combine $A_i$ and $B_i$. While doing this, we can get rid of the $O$ and $U$ matrices in $A_i$ and $B_i$, by considering
    \begin{equation*}
        A_i' := e^{J,J,J}_{-J,-J,-J} A_i = \sum_{j=i}^{J+\sigma} (A_i)^{j, m_1, m_1'}_{J, J, J} e^{j, m_1, m_1'}_{-J, -J, -J} +  (A_i)^{-j^{-}, m_1^+, m_1'^+}_{J, J, J} e^{-j^{-}, m_1^+, m_1'^+}_{-J, -J, -J}
    \end{equation*}
    and
    \begin{equation*}
        B_i' := e^{J,-J,-J-1}_{-J,-J,-J} B_i = \sum_{j=i}^{J+\sigma} (B_i)^{j,m_1,m_1'}_{J,-J,-J-1} e^{j,m_1,m_1'}_{-J,-J,-J} + (B_i)^{-j^-,m_1^+, m_1'^+}_{J,-J,-J-1} e^{-j^-,m_1^+, m_1'^+}_{-J, -J,-J}
    \end{equation*}
    which we can linearly combine to form
    \begin{align*}
        \lambda A_i' + \mu B_i' = \sum_{j=i}^{J+\sigma} & \left [ \lambda (A_i)^{j, m_1, m_1'}_{J, J, J} + \mu (B_i)^{j,m_1,m_1'}_{J,-J,-J-1} \right ] e^{j, m_1, m_1'}_{-J, -J, -J} \\
        +& \left [\lambda (A_i)^{-j^{-}, m_1^+, m_1'^+}_{J, J, J} + \mu (B_i)^{-j^-,m_1^+, m_1'^+}_{J,-J,-J-1} \right ] e^{-j^-, m_1^+, m_1'^+}_{-J, -J, -J}
    \end{align*}
    with $\lambda, \mu \in \C$, which lies in $S_J^{\circ 4}$. Using this expression, we can perform an induction argument in $i$ from above in order to separate the desired $e^{j_1, m_1, m_1'}_{-J, -J, -J}$ and $e^{-j_1^-, m_1^+, m_1'^+}_{-J, -J, -J}$ matrices from the rest.\\
    \\
    Define the subspaces $$W_i = \spann \{e^{j, m_1, m_1}_{-J, -J, -J}, \; e^{-j^-, m_1^+, m_1^+}_{-J, -J, -J}, \; J +\sigma \geq j \geq i \} \subseteq M_N(\C)$$
    for $i = {j_1, j_1 + 1, \ldots, J + \sigma}$. We claim that $W_i \subseteq S_J^{\circ 4}$ for all such $i$, which we show by first proving $W_{J+ \sigma} \subseteq S_J^{\circ 4}$ and then that $W_{i+1} \subseteq S_J^{\circ 4}$ implies $W_{i} \subseteq S_{J}^{\circ 4}$. In particular, $e^{j_2, m_2, m_2'}_{-J, -J, -J} \in S_J^{\circ 4}$, since by design either $j_1 m_1 m_1' = j_2 m_2 m_2'$ if $j_2$ is positive, or $-j_1^{-} m_1^+ m_1'^+ = j_2 m_2 m_2$ if $j_2$ is negative.\\
    \\
    The base case is $i = J+ \sigma$. For $\sigma = \frac{1}{2}$ we find that $B_{J+\sigma}' = K e^{-J, m_1^+, m_1'^+}_{-J,-J,-J} \in S_J^{\circ 4}$, and $W_{J+\sigma} = \spann \{e^{-J, m_1^+, m_1'^+}_{-J,-J,-J} \}$ at this level, proving the base case. When $\sigma = 0$, the base case is more involved. Here, we find 
    \begin{align*}
    \lambda A_{J}' + \mu B_{J}' &= \left [ \lambda (A_{J})^{J, m_1, m_1'}_{J, J, J} + \mu (B_{J})^{{J},m_1,m_1'}_{J,-J,-J-1} \right ] e^{J, m_1, m_1'}_{-J, -J, -J} \\
    &+ \left [\lambda (A_{J})^{-J^{-}, m_1^+, m_1'^+}_{J, J, J} + \mu (B_{J })^{-J ^-,m_1^+, m_1'^+}_{J,-J,-J-1} \right ] e^{-J ^-, m_1^+, m_1'^+}_{-J, -J, -J}
    \end{align*}
    for the linear combination. 
    By Lemma \ref{lem:Mlempropnr} with $j = J$ and using Equations \eqref{eq:Aicoeff} and \eqref{eq:Bicoeff}, we know that the vectors $(A_{J })^{J , m_1, m_1'}_{J, J, J}e^{J , m_1, m_1'}_{-J, -J, -J} + (A_{J })^{-J ^{-}, m_1^+, m_1'^+}_{J, J, J} e^{-J^-, m_1^+, m_1'^+}_{-J, -J, -J}$ and
    $(B_{J })^{{J },m_1,m_1'}_{J,-J,-J-1} e^{J , m_1, m_1'}_{-J, -J, -J} + (B_{J })^{-J ^-,m_1^+, m_1'^+}_{J,-J,-J-1} e^{-J^-, m_1^+, m_1'^+}_{-J, -J, -J}$ are linearly independent. Hence, we can find $\lambda_1, \mu_1$ and $\lambda_2, \mu_2$ so that $\lambda_1 A_{J }' + \mu_1 B_{J  }' = e^{J , m_1, m_1'}_{-J, -J, -J}$ and $\lambda_2 A_{J }' + \mu_2 B_{J  }' = e^{-J^-, m_1^+, m_1'^+}_{-J, -J, -J}$, i.e. $W_{J  } \subseteq S_J^{\circ 4}$.
    Thus, we have shown the base case. \\
    \\
    Now we move on to the induction step. Fix $i \in \{j_1, j_1 + 1, \ldots, J -1 + \sigma\}$, and suppose that $W_{i+1} \subseteq S_J^{\circ 4}$. Then to show that $W_i \subseteq S_J^{\circ 4}$ it suffices to show that $e^{i, m_1, m_1'}_{-J, -J, -J}, \; e^{-i^-, m_1^+, m_1'^+}_{-J, -J, -J} \in S_J^{\circ 4}$. Since $\lambda A_{i}' + \mu B_i' \in S_J^{\circ 4}$, it is enough to show that these matrices lie in $\lambda A_{i}' + \mu B_i' + W_{i+1} \subseteq S_J^{\circ 4}$. Writing this out gives
    \begin{align*}
    \lambda A_{i}' + \mu B_{i}' + W_{i+1} &= \left [ \lambda (A_{i})^{i, m_1, m_1'}_{J, J, J} + \mu (B_{i})^{i,m_1,m_1'}_{J,-J,-J-1} \right ] e^{i, m_1, m_1'}_{-J, -J, -J} \\
    &+ \left [\lambda (A_{i})^{-i^{-}, m_1^+, m_1'^+}_{J, J, J} + \mu (B_{i})^{-i^-,m_1^+, m_1'^+}_{J,-J,-J-1} \right ] e^{-i^-, m_1^+, m_1'^+}_{-J, -J, -J} + W_{i+1}
    \end{align*}
    for $\lambda, \mu \in \C$. Using Lemma \ref{lem:Mlempropnr} with $j = i$, combined with Equations \eqref{eq:Aicoeff} and \eqref{eq:Bicoeff} then gives that there are $\lambda_1, \mu_1$ and $\lambda_2, \mu_2$ such that $\lambda_1 A_{i}' + \mu_1 B_{i}' + W_{i+1} = e^{i, m_1, m_1'}_{-J, -J, -J} + W_{i+1}$ and $\lambda_2 A_{i}' + \mu_2 B_{i}' + W_{i+1} = e^{i^-, m_1^+, m_1'^+}_{-J, -J, -J} + W_{i+1}$. In particular, this shows that $W_i \subseteq S_J^{\circ 4}$. After a finite number of induction steps, we find that $W_{j_1} \subseteq S_J^{\circ 4}$, so $e^{j_2, m_2, m_2'}_{-J, -J, -J} \in S_J^{\circ 4}$, as desired. \\
    \\
    Now we deal with the edge cases. Let $j_2$ be a positive half-integer, and consider $m_2' = j_2$, so that we are dealing with the triple $j_2 m_2 j_2$. In particular, $-j_2^{-}, m_2^+, j_2^+$ is not a valid triple in $\II_S(J)$, which we will use to isolate the desired term. 
    Consider the matrix $$A_{j_2}' = e^{J, J, J}_{-J,-J, -J} \spintrunc{J+j_2, J-m_2, J-j_2} \in S_J^{\circ 4}$$
    which decomposes as
    \begin{equation*}
        A_{j_2}' = (A_{j_2})^{j_2, m_2, j_2}_{J, J, J} e^{j_2, m_2, j_2}_{-J, -J, -J} + \sum_{j={j_2+1}}^{J+\sigma} (A_{j_2})^{j, m_2, j_2}_{J, J, J} e^{j, m_2, j_2}_{-J, -J, -J} +  (A_{j_2})^{-j^{-}, m_2^+, j_2^+}_{J, J, J} e^{-j^{-}, m_2^+, j_2^+}_{-J, -J, -J}
    \end{equation*}
    with the same notation as before.
    For $j > j_2$, we know that matrices $e^{j,m_2,j_2}_{J,J,J} \in S_J^{\circ 4}$ and $e^{-j^-,m_2^+,j_2^+}_{J,J,J} \in S_J^{\circ 4}$ by the earlier part of the proof. 
    Hence, all the terms under the summation sign lie in $S_J^{\circ 4}$. Since
    \begin{equation*}
        (A_{j_2})^{j_2, m_2, j_2}_{J, J, J} = \langle f^{J+j_2}_{J - m_2, J - j_2} \psi^{j_2}_{m_2, j_2}, \psi^{J}_{J,J} \rangle \neq 0
    \end{equation*}
    by Lemma \ref{lem:maxspinorsinprod}, and $A_{j_2}' \in S_J^{\circ 4}$, we conclude that $e^{j_2, m_2, j_2}_{-J, -J, -J} \in S^{\circ 4}$ as well.\\
    \\
    In the case that $m_2' = -j_2 - 1$, we instead consider $B_{j_2}' = e^{J, -J, -J-1}_{-J, -J, -J} \spintrunc{J + j_2, -J - m_2, -J+m_2' + 1} \in S_J^{\circ 4}$. Similar to the $m_2' = j_2$ case, we have that $-j_2^{-}, m_2^+, (-j_2 -1)^+$ is not a valid triple in $\II_S(J)$. Thus, we get a similar decomposition
    \begin{align*}
        B_{j_2}' &= (B_{j_2})^{j_2, m_2, -j_2 -1}_{J, -J, -J-1} e^{j_2, m_2, -j_2 -1}_{-J, -J, -J} \\
        &+ \sum_{j=j_2 + 1}^{J+\sigma} (B_{j_2})^{j,m_2,-j_2-1}_{J,-J,-J-1} e^{j,m_2,-j_2 - 1}_{-J,-J,-J} + (B_{j_2})^{-j^-,m_2^+, (-j_2-1)^+}_{J,-J,-J-1} e^{-j^-,m_2^+, (-j_2-1)^+}_{-J, -J,-J}
    \end{align*}
    where 
    \begin{equation*}
            (B_{j_2})^{j_2, m_2, -j_2-1}_{J, -J, -J-1} = \langle f^{J+j_2}_{-J - m_2, -J + j_2+1} \psi^{j_2}_{m_2, -j_2-1}, \psi^{J}_{-J,-J-1} \rangle \neq 0
    \end{equation*}
    by Lemma \ref{lem:maxspinorsinprod}. All the terms in the sum lie in $S_{J}^{\circ 4}$, and $B_{j_2}' \in S_J^{\circ 4}$ as well, hence $e^{j_2, m_2, -j_2 -1}_{-J, -J, -J} \in S_J^{\circ 4}$. With this, we have finally shown that $e^{j_2, m_2, m_2'}_{-J, -J, -J} \in S_{J}^{\circ 4}$ for all $j_2, m_2, m_2' \in \II_S(J)$, proving the proposition.
\end{proof}

Although Theorem \ref{thm:SJpropnr} only gives bounds on the propagation number, we can calculate the propagation number numerically for small $J$. It turns out that already for the smallest operator system $S_0$ we find $\prop(S_0) = 4$. For other small values of $J$ the propagation number can also be determined by calculation, and we find for $J = 0, \frac{1}{2}, 1$ that $\prop(S_J) = 4$ using Python code\footnote{Available at \url{https://github.com/MathVerhoeven/Spectral-Truncation-SU2/}}. Based on this, we put forward the following conjecture.
\begin{conjecture}
    For all $J \in \frac{1}{2}\N$ it holds that $\prop(S_J) = 4$.
\end{conjecture}

\section{Spectral localizer applications}
In this section, we use the spectral localizer to numerically calculate $K_1$-classes of irreducible representations $\pi_j$ of $\SU(2)$, where we view $\pi_j \in M_{2j+1}(C(\SU(2)))$ as a matrix-valued function on $\SU(2)$. The spectral localizer introduced by Loring and Schulz-Baldes allows for the computation of the index map using only a finite amount of spectral information, in this case by working in some operator system $S_J$ as introduced in the previous chapter. The spectral localizer works by choosing an appropriately small tuning parameter, and a sufficiently large truncation parameter. We compare for what parameters we find the $K_1$-classes of $\pi_j$ numerically, which is the specific situation considered in this thesis, compared to the parameters given in \cite{SchulzBaldes2023}, which hold for a more general class of situations.\\
\\
Before discussing the spectral localizer, we record some facts about the $K$-theory of $A := C(\SU(2))$. The mapping $\Phi \colon S^3 \to \SU(2)$ given by \begin{equation*}
	\Phi(x_1,x_2,x_3,x_4) = \begin{pmatrix}
		x_4 + ix_1 & x_2 + ix_3 \\
		-x_2 + ix_3 & x_4 - ix_1
	\end{pmatrix}
\end{equation*} 
is a homeomorphism, and thus $A = C(\SU(2)) \cong C(S^3)$. Hence, we know that $K_1(C(\SU(2))) \cong K_1(S^3) \cong \Z$, which follows from a well-known argument using $S^3 = \R^3 \cup \{\infty\}$ and Bott periodicity \cite{Rordam2000, SchulzBaldesSpheres2023}. This $K_1$ group turns out to be generated by the defining representation of $\SU(2)$.
\begin{proposition}
    The unitary $\pi_{\frac{1}{2}} \in M_2(A)$ generates $K_1(A)$.
\end{proposition}
\begin{proof}
    In \cite{SchulzBaldesSpheres2023}, a generator of $K_1(C(S^{3})) \cong \Z$ is given by
    \begin{equation*}
        h(g) = \begin{pmatrix}
        a & b \\ b^* & -a^*
        \end{pmatrix}
    \end{equation*}
    after translating $S^3$ to $\SU(2)$ coordinates with the homeomorphism $\Phi \colon S^3 \to \SU(2)$.
    This $h \in M_2(A)$ is homotopic to $$\pi_{\frac{1}{2}}(g) = \begin{pmatrix}
        a & b \\ -b^* & a^*
        \end{pmatrix}$$ by the path of unitaries
    \begin{equation*}
        \gamma(t)(g) := \begin{pmatrix}
            e^{2 \pi it} a & e^{2 \pi i t} b \\
            -e^{\pi it} b^* & e^{\pi i t} a^*
        \end{pmatrix}
    \end{equation*} 
    with $t \in [0,1]$, satisfying $\gamma(0) = \pi_{\frac{1}{2}}$ and $\gamma(1) = h$, so that $[h]_1 = [\pi_{\frac{1}{2}}]_1$ as desired. \\
    \\
    Alternatively, one could use the index map homomorphism. For the quantum group $\SU_q(2)$ this is done in \cite{VanSuijlekom2005}.
\end{proof}
In analogy with the circle, where the spectral localizer on $e^{imt}$ is well-behaved, we will apply the spectral localizer to the irreducible representations of $\SU(2)$.
For each $j \in \frac{1}{2}\N$, one can view $\pi_j \colon \SU(2) \to V_j$ as a unitary $2j+1$ by $2j+1$ matrix over $A$. Thus, $\pi_j$ induces a class in $K_1(A) \cong \Z$, which we can calculate explicitly using Clebsch-Gordan theory.
\begin{proposition}
    For $j \in \frac{1}{2}\N$ it holds that
    \begin{equation}\label{eq:indexrep}
        [\pi_j]_1 = \frac{1}{6}(2j)(2j+1)(2j+2)\cdot [\pi_{\frac{1}{2}}]_1.
    \end{equation}
\end{proposition}
\begin{proof}
    We show this by induction on $j = \frac{n}{2}$. The base case $[\pi_0]_1 = [1]_1 = 0$ is clear, and so is $[\pi_{\frac{1}{2}}]_1 = 1 \cdot [\pi_{\frac{1}{2}}]_1$. Let $n \in \N$, $n \geq 1$ and suppose that Equation \eqref{eq:indexrep} holds for $\pi_{\frac{n}{2}}$ and $\pi_{\frac{n-1}{2}}$. We will show that it also holds for $\pi_{\frac{n+1}{2}}$.
    We claim that the following unitary matrices define the same classes in $K_1(A)$:
    \begin{equation}\label{eq:K1equality} \left [
        \begin{pmatrix}
        \pi_{\frac{1}{2}} \tensor 1_{n+1} & 0 \\
        0 & 1_2 \tensor \pi_{\frac{n}{2}}
        \end{pmatrix}
        \right ]_1 = \left [
        \begin{pmatrix}
            \pi_{\frac{1}{2}} \tensor \pi_{\frac{n}{2}} & 0 \\
            0 & 1_2 \tensor 1_{n+1}
        \end{pmatrix}
        \right ] = \left [
        \begin{pmatrix}
            \pi_{\frac{n+1}{2}} \oplus \pi_{\frac{n-1}{2}} & 0 \\
            0 & 1_{2(n+1)}
        \end{pmatrix}
        \right ]_1.
    \end{equation}
    The first equality is due to Whitehead's lemma (\cite{Rordam2000}, Lem 2.1.5), which states that
    \begin{equation*}
        \left [\begin{pmatrix}
            u & 0 \\ 0 & v
        \end{pmatrix} \right ]_1 = \left [  \begin{pmatrix}
            uv & 0 \\ 0 & 1
        \end{pmatrix}\right ]_1
    \end{equation*}
    for unitaries $u,v$. Taking $u = \pi_{\frac{1}{2}} \tensor 1_{n+1}$ and $v = 1_2 \tensor \pi_{\frac{n}{2}}$ gives the desired equality. The second equality follows from the Clebsch-Gordan intertwiner $C^{\frac{1}{2}, \frac{n}{2}} \colon V_{\frac{1}{2}} \tensor V_{\frac{n}{2}} \to V_{\frac{n+1}{2}} \oplus V_{\frac{n-1}{2}}$. Specifically, we have that
    \begin{equation*}
        \begin{pmatrix}
            C^{\frac{1}{2}, \frac{n}{2}} & 0 \\ 0 & 1
        \end{pmatrix}
    \end{equation*}
    is a unitary matrix such that
    \begin{equation*}
                \begin{pmatrix}
            \pi_{\frac{n+1}{2}} \oplus \pi_{\frac{n-1}{2}} & 0 \\
            0 & 1_{2(n+1)}
        \end{pmatrix} = 
                \begin{pmatrix}
            C^{\frac{1}{2}, \frac{n}{2}} & 0 \\ 0 & 1
        \end{pmatrix}
                \begin{pmatrix}
            \pi_{\frac{1}{2}} \tensor \pi_{\frac{n}{2}} & 0 \\
            0 & 1_2 \tensor 1_{n+1}
        \end{pmatrix}
                \begin{pmatrix}
            C^{\frac{1}{2}, \frac{n}{2}} & 0 \\ 0 & 1
        \end{pmatrix}^*,
    \end{equation*}
    so that the $K_1$-classes satisfy
    \begin{equation*}
        \left [ \begin{pmatrix}
            \pi_{\frac{n+1}{2}} \oplus \pi_{\frac{n-1}{2}} & 0 \\
            0 & 1_{2(n+1)}
        \end{pmatrix} \right ]_1 = [C^{\frac{1}{2}, \frac{n}{2}}]_1 + \left [ \begin{pmatrix}
            \pi_{\frac{1}{2}} \tensor \pi_{\frac{n}{2}} & 0 \\
            0 & 1_2 \tensor 1_{n+1}
        \end{pmatrix} \right ]_1 - [C^{\frac{1}{2}, \frac{n}{2}}]_1
    \end{equation*}
    establishing the second equality. \\
    \\
    To simplify the $K_1$-classes on the left and right-hand sides of Equation \eqref{eq:K1equality}, we use $[u \oplus v]_1 = [u]_1 \oplus [v]_1$ which implies that $[u \tensor 1_k]_1 = k[u]$ by induction. Hence, the earlier equality gives $(n+1)[\pi_{\frac{1}{2}}]_1 + 2 [\pi_{\frac{n}{2}}]_1 = [\pi_{\frac{n+1}{2}}]_1 + [\pi_{\frac{n-1}{2}}]_1$. By the induction hypothesis, we find
    \begin{align*}
        [\pi_{\frac{n+1}{2}}]_1 &= \left (n+1 + \frac{1}{3}n(n+1)(n+2) - \frac{1}{6}(n-1)n(n+1) \right ) [\pi_{\frac{1}{2}}]_1\\
        &= \frac{1}{6}(n+1)(n+2)(n+3)[\pi_{\frac{1}{2}}]_1
    \end{align*}
    which yields the desired result after substituting $n = 2j$.
\end{proof}

The spectral localizer is a method of calculating the index pairing map using a finite matrix. We first discuss the spectral localizer in some generality before applying it in the specific case $A = C(G)$.\\
\\
In this section, $(\mathcal A, \mathcal H, D)$ will denote a spectral triple and the $C^*$-algebra $A := \overline{\mathcal A}$ will denote the closure of $\mathcal A$ in $B(\mathcal H)$. Since $\SU(2)$ is an odd-dimensional manifold, we are interested in the odd index map.
\begin{definition}[\cite{Higson2001}, Prop 8.7.1; \cite{GraciaBonda2001} p. 485; \cite{CPR2011}, Section 4.3]
    Let $U \in M_n(\mathcal A)$ be unitary. Then the \textbf{odd index pairing} of $U$ (with $D$) is given by
    \begin{equation}
        \langle [U]_1, [D] \rangle  = \ind(P_+^{\oplus n} U P_+^{\oplus n} \colon \; P_+^{\oplus n} H^{\oplus n} \to P_+^{\oplus n} H^{\oplus n} )
    \end{equation}
    where $P_+ = \chi_{(0,\infty)}(D)$ and where $\ind(T)$ denotes the Fredholm index of $T \colon \mathcal H \to \mathcal H$.
\end{definition}

To calculate a Fredholm index for some operator $T$ on $\mathcal H$, we need to know how $T$ acts on an infinite amount of vectors, which is simply impossible on a computer. However, this is remarkably not the case for the index map, where we can use the spectral localizer.
We use Chapter 10 of the book \cite{SchulzBaldes2023} as the main source for the spectral localizer.
The spectral localizer takes a different form in the even $K_0$ and odd $K_1$ case, but since $\SU(2)$ is an odd-dimensional manifold, we will focus on the odd spectral localizer.\\
\\
For the spectral localizer, we use spectral truncations of a spectral triple $(\mathcal A, \mathcal H, D)$, for which we introduce some notation. We denote by $P_{\rho} = \chi_{[-\rho, \rho]}(D)$ the finite-dimensional spectral projection for $\rho \geq 0$, which is analogous to $P_J$ as defined for $(C^{\infty}(G), L^2(S), D)$. 
For notational convenience, we will write $\mathcal H_\rho = P_\rho \mathcal H$ and $T_\rho = P_\rho T P_\rho$ for the truncation of a (potentially unbounded) operator $T$ on $\mathcal H$. Since $D$ is self-adjoint with compact resolvent, the spectrum of $D$ consists of real eigenvalues, satisfying $\lim_{n \to \infty} |\lambda_n| = \infty$, each with finite multiplicity\footnote{The cited results in this section additionally assume that $0$ does not lie in the spectrum of $D$, which is the case for $\SU(2)$. This assumption can be removed for the odd spectral localizer (\cite{SchulzBaldes2023}, p. 314; \cite{vdDungen2026}).}. Thus, $\mathcal H_{\rho}$ is a finite dimensional Hilbert space, and $T_\rho$ is a matrix on the space. In the case of $\SU(2)$, this is the same as $P_J L^2(S)$ and $\spintrunc[J]{T}$ for $\rho = \frac{3}{2} + 2J$.
\begin{definition}[\cite{SchulzBaldes2023}, Section 10.4]\label{def:speclocSB}
    Let $(\mathcal A, \mathcal H, D)$ be a spectral triple, and let $x \in \mathcal A$ be an invertible operator. Let $\kappa > 0$ be the tuning parameter and $\rho > 0$ the truncation parameter. Then the odd spectral localizer on $\mathcal H_{\rho} \oplus \mathcal H_{\rho}$ is defined by
    \begin{equation*}
        L^{odd}_{\kappa} = \begin{pmatrix}
        \kappa D & x \\
        x^* & -\kappa D
        \end{pmatrix}
    \end{equation*}
    acting on $\mathcal H \oplus \mathcal H$.
    The truncated odd spectral localizer is
    \begin{equation*}
        L^{odd}_{\kappa, \rho} = \begin{pmatrix}
        \kappa D_\rho & x_\rho \\
        x_\rho^* & -\kappa D_\rho
        \end{pmatrix}
    \end{equation*}
    acting on $\mathcal H_\rho \oplus \mathcal H_{\rho}$.
\end{definition}
The spectral localizer realizes the index by using the signature.
\begin{definition}
    Let $T \in M_n(\C)$ be a finite-dimensional matrix. Then the \textbf{signature} of $T$, denoted by $\Sig(T)$, is given by the number of strictly positive eigenvalues of $T$ minus the number of strictly negative eigenvalues, counting multiplicities.
\end{definition}

\begin{theorem}[Spectral localizer theorem, \cite{SchulzBaldes2023} Thm 10.4.1]\label{thm:SpecLocSB}
    Let $g = \|x^{-1}\|^{-1}$ be the spectral gap of the invertible operator $x \in \mathcal A$, and let $P_+ = \chi_{(0,\infty)}(D)$ denote the positive spectral projection of $D$. Suppose that
    \begin{equation}\label{eq:SLkapparho}
        \kappa \leq \frac{g^3}{12\|x\|\|[D,x]\|}, \: \phantom{=} \frac{2g}{\kappa} \leq \rho.
    \end{equation}
    Then the matrix $L^{odd}_{\kappa, \rho}$ satisfies the bound $L^{odd}_{\kappa, \rho} \geq \frac{g^2}{4}1_\rho^{\oplus 2}$. In particular, $L^{odd}_{\kappa, \rho}$ is invertible and thus has a well-defined signature $\Sig(L^{odd}_{\kappa, \rho})$. It is independent of $\kappa$ and $\rho$ satisfying Equation \eqref{eq:SLkapparho}, and
    \begin{equation*}
        \langle [u]_1, [D]\rangle = \ind(P_+uP_+) = \frac{1}{2}\Sig(L^{odd}_{\kappa, \rho})
    \end{equation*}
    where $u = x|x|^{-1}$ is the unitary phase of $x$.
\end{theorem}

In particular, the index of $P_+ u P_+$ may be difficult to compute, since it requires finding the kernel of infinite-dimensional operators. By using the spectral localizer, this problem reduces to finding the signature of a matrix, which is more computationally feasible. 
We will investigate the behaviour of $K_1$-classes of unitary representations $\pi_j$, viewed as unitary operators in $M_{2j+1}(\mathcal A)$.
However, the spectral localizer as presented in \cite{SchulzBaldes2023} can a priori only be applied for $x \in \mathcal A$, and not for matrix-valued $X \in M_n(\mathcal A)$. Hence, in order to analyse $[\pi_{j}]_1 \in K_1(A)$ using the spectral localizer, we need to extend the formalism to matrix-valued operators, which is straightforward.
\begin{proposition}
    For $(\mathcal A, \mathcal H, D)$ a spectral triple, also $$(\mathcal A \tensor \End(V_j), \mathcal H \tensor V_j, D \tensor 1_{V_j}) \cong (M_n(\mathcal A), \mathcal H^{\oplus n}, D^{\oplus n})$$
    is a spectral triple, with $n = 2j+1$. 
\end{proposition}
We omit the proof, since it consists of standard checks, and note that this proposition is alluded to in several sources (\cite{VanSuijlekom2025} Remark 6.22, \cite{Khalkhali2009} p. 187, \cite{CPR2011} Section 4.3). 
To use the spectral localizer on this spectral triple, we need to understand its spectral truncations. For this it suffices to know the following.
\begin{proposition}
    Let $S \subseteq \R$ be a (measurable) set and let $P_S = \chi_S(D)$, $Q_S = \chi_S(D^{\oplus n})$ be spectral projections of $D$ and $D^{\oplus n}$ on $S$. Consider truncations $P_\rho = \chi_{[-\rho, \rho]}(D)$ and $Q_{\rho} = \chi_{[-\rho, \rho]}(D^{\oplus n})$. Then $Q_S = P_S^{\oplus n}$ as projections, $Q_\rho M_n(\mathcal A) Q_\rho = M_n(P_\rho \mathcal A P_\rho)$ as operator systems, $Q_\rho \mathcal H^{\oplus n} = (P_{\rho} \mathcal H)^{\oplus n}$ as Hilbert spaces and $Q_\rho D^{\oplus n} Q_{\rho} = (P_{\rho} D P_{\rho})^{\oplus n}$ as truncated Dirac operators.
\end{proposition}
\begin{proof}
    This follows from a direct calculation, and is omitted.
\end{proof}
In particular, this means that the truncated operator system spectral triple of $(M_n(\mathcal A), \mathcal H^{\oplus n}, D^{\oplus n})$ is given by $(M_n(\mathcal A_\rho), \mathcal H_\rho^{\oplus}, D_\rho^{\oplus n})$, as expected. Applying the spectral localizer theorem to the matrix spectral triple gives the following proposition.
\begin{proposition}\label{prop:SpecLocMn}
    Let $n \in \N$, and let $g = \|X^{-1}\|^{-1}$ be the spectral gap of the invertible operator $X \in M_n(\mathcal A)$. Suppose that
    \begin{equation}\label{eq:MnSLkapparho}
        \kappa \leq \frac{g^3}{12\|X\|\|[D^{\oplus n},X]\|}, \: \phantom{=} \frac{2g}{\kappa} \leq \rho.
    \end{equation}
    Then the matrix
    \begin{equation*}
        L_{\kappa, \rho,n}^{odd} = \begin{pmatrix}
        \kappa D_\rho^{\oplus n} & X_\rho \\
        X_\rho^* & -\kappa D_\rho^{\oplus n} 
        \end{pmatrix}
    \end{equation*}
    satisfies the bound $L^{odd}_{\kappa, \rho,n} \geq \frac{g^2}{4}1_\rho^{\oplus 2n}$. In particular, $L^{odd}_{\kappa, \rho,n}$ is invertible and thus has a well-defined signature $\Sig(L^{odd}_{\kappa, \rho,n})$. It is independent of $\kappa$ and $\rho$ satisfying Equation \eqref{eq:MnSLkapparho}, and
    \begin{equation*}
        \langle [U]_1, [D] \rangle = \ind(P_+^{\oplus n}UP_+^{\oplus n}) = \frac{1}{2}\Sig(L^{odd}_{\kappa, \rho,n})
    \end{equation*}
    where $U = X|X|^{-1}$ is the unitary phase of $X$.
\end{proposition}
\setcounter{subsection}{0}
\subsection{The spectral localizer on $\SU(2)$}
\setcounter{subsection}{9}
In this section we investigate the behaviour of the spectral localizer on the spectral triple \phantom{==} $(C^{\infty}(\SU(2)), L^2(S), D)$. 
In several examples, the spectral localizer works in a larger range than predicted by Theorem \ref{thm:SpecLocSB}. For example, on the circle, the spectral localizer applied to a unitary operator $u \in C(S^1)$ already finds the index for $\kappa \leq \frac{2\sqrt{3} - 3}{\|[D,u]\|}$ instead of $\kappa \leq \frac{1}{12\|[D,u]\|}$, and in the special case $u = e^{imt}$ it suffices to take $\kappa < \frac{2}{|m|}$ and $|m| \leq \rho$, independent of $\kappa$ \cite{Hekkelman2021}. Also in applications of the spectral localizer to solid state physics, it has been observed that $\kappa$ can be chosen larger than expected \cite{Cerjan2024}. In an attempt to better understand the range of the spectral localizer, we consider the spectral localizer applied to $\SU(2)$ in this section.\\
\\
We will only consider the spectral triple $(C^{\infty}(G), L^2(S), D)$ and hence we keep the second slot in the index pairing fixed. The odd index pairing is in this case a well-defined group homomorphism $K_1(C(G)) \cong \Z \to \Z$, since $C^{\infty}(G)$ is closed under holomorphic functional calculus (\cite{GraciaBonda2001}, Section 3.8). This allows for the direct calculation of $K_1$-classes $[\pi_j]_1$, once the value of the index pairing applied to the generator $\langle [\pi_{\frac{1}{2}}]_1, [D] \rangle$ is known. We will find these classes by using the spectral localizer, but it turns out that we will find this for a much larger range of $\kappa$ and $\rho$ values than given in Theorem \ref{thm:SpecLocSB}. \\
\\
For unitaries $\pi_j$ the odd spectral localizer is given by
\begin{equation*}
    L_{\kappa, \rho}(\pi_j) := \begin{pmatrix}
    \kappa D_\rho^{\oplus 2j + 1} & (\pi_j)_\rho \\ (\pi_j)_\rho^* & -\kappa D_{\rho}^{\oplus 2j + 1}
    \end{pmatrix} = \begin{pmatrix}
    \kappa \spintrunc{D}^{\oplus 2j + 1} & \spintrunc{\pi_j} \\
    \spintrunc{\pi_j}^* & - \kappa \spintrunc{D}^{\oplus 2j + 1}
    \end{pmatrix}
\end{equation*}
for a tuning parameter $\kappa$ and a truncation parameter $\rho = \frac{3}{2} + 2J$, where
\begin{equation*}
    \spintrunc{\pi_j} = \sqrt{\frac{1}{2j+1}} \begin{pmatrix}
    \spintrunc{j,j,j} & \dots & \spintrunc{j,-j,j} \\
    \vdots & \ddots & \vdots \\
    \spintrunc{j,j,-j} & \dots & \spintrunc{j,-j,-j}
    \end{pmatrix}
\end{equation*}
as a block matrix with $\spintrunc{jmm'} \in S_J$ the spectral truncation of $f^{j}_{m,m'}$. 
We can calculate the signature of $L_{\kappa, \rho}(\pi_j)$ numerically using Python, and by Proposition \ref{prop:SpecLocMn} we know that this gives the index map $\langle [\pi_j]_1, [D] \rangle = \frac{1}{2}\Sig(L_{\kappa, \rho}(\pi_j))$ for appropriate values of $\kappa$ and $\rho$.
In Figure \ref{fig:SpecLocplots} we present $\frac{1}{2}\Sig(L_{\kappa, \rho}(\pi_j))$ as a function of the tuning parameter $\kappa$ for several values of $\rho$ and $j$.
\newpage
\begin{figure}[H]
     \centering
     \begin{subfigure}[b]{0.45\textwidth}
         \centering
         \includegraphics[width=\textwidth]{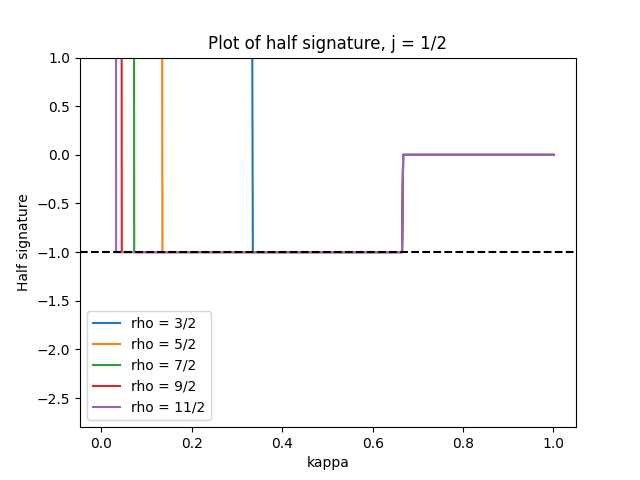}
         \caption{$j=\frac{1}{2}$}
         \label{fig:SLj1/2}
     \end{subfigure}
     \hfill
     \begin{subfigure}[b]{0.45\textwidth}
         \centering
         \includegraphics[width=\textwidth]{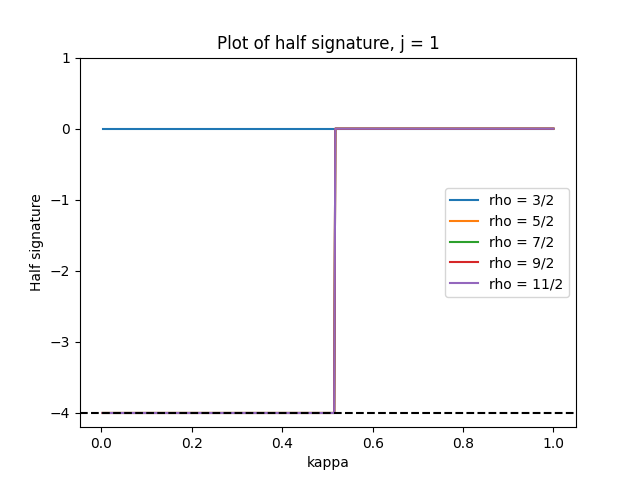}
         \caption{$j=1$}
         \label{fig:SLj1}
     \end{subfigure}
     \vfill
     \begin{subfigure}[b]{0.45\textwidth}
         \centering
         \includegraphics[width=\textwidth]{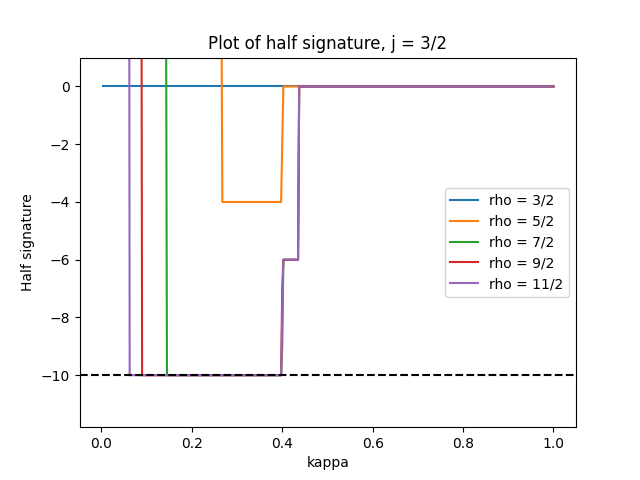}
         \caption{$j=\frac{3}{2}$}
         \label{fig:SLj3/2}
     \end{subfigure}
     \hfill
     \begin{subfigure}[b]{0.45\textwidth}
         \centering
         \includegraphics[width=\textwidth]{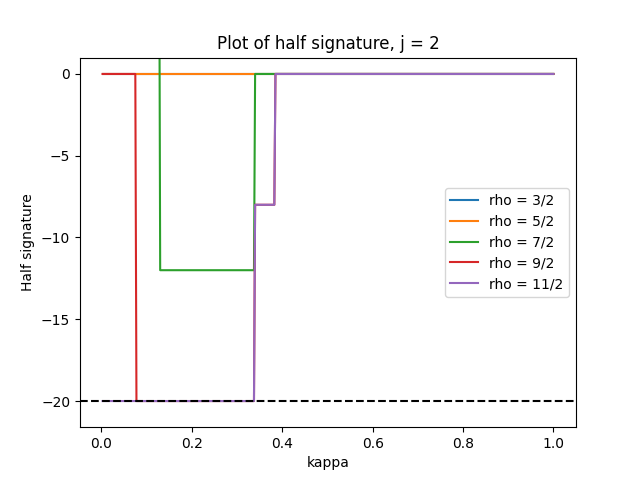}
         \caption{$j=2$}
         \label{fig:SLj2}
     \end{subfigure}
     \vfill
     \begin{subfigure}[b]{0.45\textwidth}
         \centering
         \includegraphics[width=\textwidth]{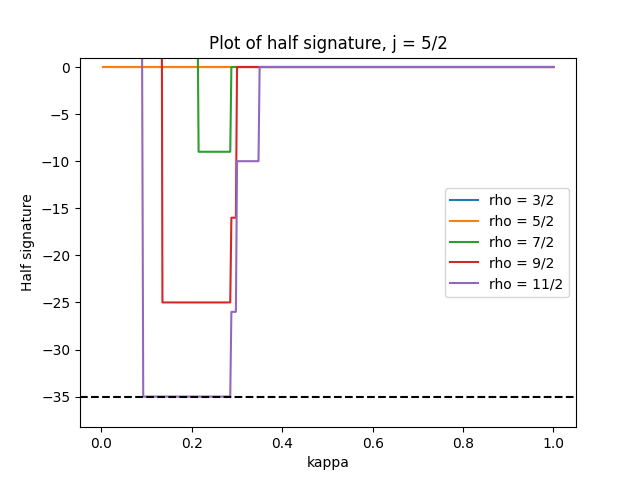}
         \caption{$j=\frac{5}{2}$}
         \label{fig:SLj5/2}
     \end{subfigure}
    \hfill
     \begin{subfigure}[b]{0.45\textwidth}
         \centering
         \includegraphics[width=\textwidth]{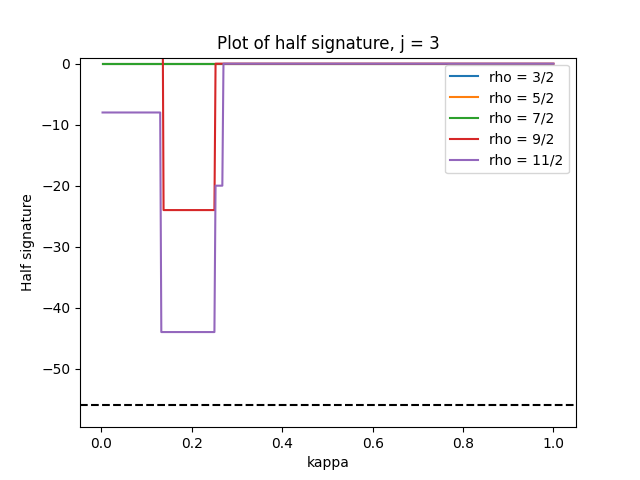}
         \caption{$j=3$}
         \label{fig:SLj3}
     \end{subfigure}
        \caption{The value of $\frac{1}{2}\Sig(L_{\kappa, \rho}(\pi_j))$ as a function of $\kappa$ for several values of $\rho$ and $j$. The analytically computed index appears as a horizontally dotted line. Code producing these plots is available at \url{https://github.com/MathVerhoeven/Spectral-Truncation-SU2/}}
        \label{fig:SpecLocplots}
\end{figure}
\newpage

Several types of behaviour can be read off from these plots, at least for low values of $j$. Firstly, it appears that the index of $\pi_j$ can be determined for $\rho \geq \frac{3}{2} + 2(j-\frac{1}{2})$, which amounts to performing calculations starting in the operator system $S_{j-\frac{1}{2}}$. As such, it is expected that the plot for $j = 3$ in Figure \ref{fig:SLj3} does not capture the index, as the highest truncation level plotted is $\rho = 11/2 < 13/2$. Furthermore, the highest value of the truncation parameter $\kappa$ for which the index is found seems to be given by $\kappa = \frac{1}{j+1}$, independently of the truncation level. \\
\\
These should be compared to the bounds on $\rho$ and $\kappa$ obtained from Proposition \ref{prop:SpecLocMn}. Since these bounds hold in a far more general setting, and do not take the representation theory of the situation into account, it is natural for these bounds to be worse. Nevertheless, it is interesting to see the discrepancy in a specific scenario such as the irreducible representations of $\SU(2)$. Since $\pi_j$ is unitary, the spectral gap and norm satisfy $g = 1$ and $\|\pi_j\| = 1$ respectively, which simplifies the bounds in Equation \eqref{eq:MnSLkapparho} to
\begin{equation}\label{eq:unitarykapparho}
        \kappa \leq \frac{1}{12\|[D^{\oplus 2j+1},\pi_j]\|}, \: \phantom{=} \frac{2}{\kappa} \leq \rho.
\end{equation}
Since we are dealing with matrix valued functions, it is somewhat difficult to determine the norm of the commutator $\|[D^{\oplus 2j + 1}, \pi_j]\|$ exactly. However, we can still calculate a lower bound by simply evaluating the commutator on a certain spinor. By taking
\begin{equation*}
    \Psi = \begin{pmatrix}
    \psi^{+0}_{0,0} \\ 0 \\ \vdots \\ 0
    \end{pmatrix} \in L^2(S) \tensor \C^{2j+1}
\end{equation*}
with $\|\Psi\| = 1$, we can explicitly calculate that $\|[D^{\oplus 2j+1}, \pi_j] \Psi\| = \sqrt{4j^2 + 6 j + 2} \frac{\sqrt{2j}}{\sqrt{2j+1}}$, 
which implies that certainly $\|[D^{\oplus 2j+1}, \pi_j]\| \geq \sqrt{4j^2 + 6 j + 2} \frac{\sqrt{2j}}{\sqrt{2j+1}}$. With this lower bound, it follows that 
\begin{equation*}
    \kappa \leq \frac{1}{12\|[D^{\oplus 2j+1},\pi_j]\|} \leq \frac{\sqrt{2j+1}}{12\sqrt{2j}\sqrt{4j^2 + 6j + 2}}
\end{equation*} is a (possibly quite weak) predicted upper bound for the tuning parameter. For $j = \frac{1}{2}, 1, \frac{3}{2}, 2$ this gives values of approximately $0.049$, $0.030$, $0.022$ and $0.017$, which are at least an order of magnitude smaller than the values $\kappa = \frac{1}{j+1}$ obtained from the figures. As a consequence, the truncation level $\rho$ predicted by Equation \eqref{eq:unitarykapparho} is far higher than what turns out to be necessary. For example, for $j = \frac{1}{2}$ one would need at least $\rho > 41.5$, corresponding to the operator system $S_{20}$. Elements of the operator system $S_{20}$ are 49364 by 49364 matrices, which is much larger than the 4 by 4 matrices used in $S_0$ to determine the same index. Already for $\pi_2$ Equation \eqref{eq:unitarykapparho} would predict the use of matrices with side length larger than 1 million, as elements of $S_{58}$.\\
\\
On the other hand, the lowest working values of $\kappa$ show more subtle behaviour. For the half-integer values of $j$ investigated, there is a cutoff value of $\kappa$, so that $$\lim_{\kappa \to 0^+} \frac{1}{2}\Sig(L_{\kappa, \rho}(\pi_j)) \neq \langle [\pi_{j}]_1, [D] \rangle.$$ Furthermore, this cutoff value depends on $\rho$ in the sense that higher truncations (more information) yield more eligible values of $\kappa$ (less tuning-dependent). However, for integer $j$, the story is somewhat different. For $j = 1$ we see in Figure \ref{fig:SLj1} that $$\lim_{\kappa \to 0^+} \frac{1}{2} \Sig(L_{\kappa, \rho}(\pi_1)) = \langle [\pi_1]_1, [D] \rangle$$ for all values of $\rho$ for which the index is successfully found. On the other hand, for $j =2$ in Figure \ref{fig:SLj2} we see the same behaviour for $\rho = \frac{11}{2}$, but not for $\rho = \frac{9}{2}$ even though the index can already be determined at $\rho = \frac{9}{2}$.\\
\\
This behaviour has implications for the definition of the index map on operator system $K$-theory defined in \cite{VanSuijlekom2024}. There, an index mapping $\ind_2^{\delta}$ for an operator system $S$ is defined for $\delta > 0$ by
\begin{equation*}
    \ind^\delta_2(x) := \lim_{\kappa, s \to 0} \frac{1}{4}\Sig(\mathcal L_{\kappa}(D,x,s))
\end{equation*}
for $x \in S$ satisfying some notion of invertibility up to $\delta$, where $\mathcal L_{\kappa}(D,x,s)$ is a generalized spectral localizer. For truncated spectral triples $(P_\rho \mathcal A P_\rho, P_\rho \mathcal H, P_\rho D P_\rho)$ this spectral localizer is shown to reduce to $(L_{\kappa,\rho}^{odd}(x))^{\oplus 2}$ as defined in Definition \ref{def:speclocSB} for $x \in \mathcal A$ invertible, by taking $s = 0$. Then $\delta$ can be picked appropriately, and one finds the definition
\begin{equation*}
    \ind_2(x) := \lim_{\kappa \to 0}\frac{1}{2} \Sig(L_{\kappa, \rho}^{odd}(x))
\end{equation*} 
for this index mapping on a truncated spectral triple. However, based on the plots in Figure \ref{fig:SpecLocplots} for half-integer values of $j$, we predict that this limit does not recover the usual index mapping for $\mathcal A = C^{\infty}(\SU(2))$. On the circle it turns out that there are no lower bounds on $\kappa$, which motivated this definition of $\ind_2$ in the first place.
\\
\\
Another comparison can be made to the analysis of the spectral localizer of the circle $S^1$ in \cite{Hekkelman2021}. An important difference of the circle compared to $\SU(2)$ is that the circle is a commutative Lie group, and hence all irreducible representations are one-dimensional, given by $e^{int}$ for $n \in \Z$. The relevant bounds of $\kappa$ and $\rho$ are explicitly determined in terms of $n$ by writing out the spectral localizer and using properties of Toeplitz matrices, i.e. truncations of $e^{int}$. Furthermore, in $K_1(C(S^1)) \cong \Z$ it holds that $[e^{int}]_1 = -n$, and as such any function $f \in C(S^1)$ with $0 \notin f(S^1)$ is homotopic to $e^{int}$ for some $n$. With this in mind, the bounds on $\kappa$ and $\rho$ can be extended and modified to give similar bounds for the spectral localizer applied to any unitary $u \in C^{\infty}(S^1)$. \\
\\
For $C(\SU(2))$, the explicit determination of these bounds is more difficult. One reason is that the structure of Toeplitz matrices is easier to describe compared to the structure of truncations $\spintrunc{jmm}$. As such, while the eigenvalues of the spectral localizer are algebraically computable for the circle, this is more difficult to achieve for $\SU(2)$.  
Even if this were to be done, it holds that the indices of representations are not all of $K_1(A)$, i.e. $\{[\pi_j]_1\}_{j \in \frac{1}{2}\N} \neq \Z$, compared to the case on the circle where $\{[e^{int}]_1\}_{n \in \Z} = \Z$ does hold.
As such, it is not necessarily true that a (matrix-valued) function is homotopic to one of the $\pi_j$, which obstructs the extension of $\kappa$ and $\rho$ bounds found for $\pi_j$ to arbitrary matrix-valued functions. An alternative would be to find bounds for determining the index of $\pi_j^{n}$, which has index $-n$. Regardless, even without these explicit bounds it is clear that the indices of irreducible representations of $S^1$ and $\SU(2)$ are found in a larger range of $\kappa$ and $\rho$ than the spectral localizer theorem predicts.

\bibliographystyle{alpha} 
\bibliography{reference_items.bib}

@book{Duistermaat2000,
  title = {Lie Groups},
  ISBN = {9783642569364},
  ISSN = {2191-6675},
  url = {http://dx.doi.org/10.1007/978-3-642-56936-4},
  DOI = {10.1007/978-3-642-56936-4},
  journal = {Universitext},
  publisher = {Springer Berlin Heidelberg},
  author = {Duistermaat,  Hans and Kolk, Johan},
  year = {2000}
}

@BOOK{Bohm2001,
  title     = "Quantum Mechanics: Foundations and Applications",
  author    = "Bohm, Arno",
  publisher = "Springer",
  series    = "Theoretical and Mathematical Physics",
  month     =  jul,
  year      =  2001,
  address   = "New York, NY",
  language  = "en",
  edition = 3
}

@article {Koornwinder1981,
    AUTHOR = {Koornwinder, Tom},
     TITLE = {Clebsch-{G}ordan coefficients for {${\rm SU}(2)$}\ and {H}ahn
              polynomials},
   JOURNAL = {Nieuw Arch. Wisk. (3)},
  FJOURNAL = {Nieuw Archief voor Wiskunde. Derde Serie},
    VOLUME = {29},
      YEAR = {1981},
    NUMBER = {2},
     PAGES = {140--155},
      ISSN = {0028-9825},
   MRCLASS = {22E70 (33A75 81C40)},
  MRNUMBER = {641673},
MRREVIEWER = {Ji\v r\'i\ Patera},
}

@book{Paulsen2003,
  title = {Completely Bounded Maps and Operator Algebras},
  ISBN = {9780511546631},
  url = {http://dx.doi.org/10.1017/CBO9780511546631},
  DOI = {10.1017/cbo9780511546631},
  publisher = {Cambridge University Press},
  author = {Paulsen,  Vern},
  year = {2003},
  month = feb 
}

@article{ConnesvS2020,
  title = {Spectral Truncations in Noncommutative Geometry and Operator Systems},
  volume = {383},
  ISSN = {1432-0916},
  url = {http://dx.doi.org/10.1007/s00220-020-03825-x},
  DOI = {10.1007/s00220-020-03825-x},
  number = {3},
  journal = {Communications in Mathematical Physics},
  publisher = {Springer Science and Business Media LLC},
  author = {Connes,  Alain and van Suijlekom,  Walter D.},
  year = {2020},
  month = jul,
  pages = {2021–2067}
}

@book{Hall2015,
  title = {Lie Groups,  Lie Algebras,  and Representations: An Elementary Introduction},
  ISBN = {9783319134673},
  ISSN = {2197-5612},
  url = {http://dx.doi.org/10.1007/978-3-319-13467-3},
  DOI = {10.1007/978-3-319-13467-3},
  journal = {Graduate Texts in Mathematics},
  publisher = {Springer International Publishing},
  author = {Hall,  Brian C.},
  year = {2015}
}

@article{Homma2000,
  title = {A Representation of $Spin(4)$ on the Eigenspinors of the Dirac Operator on $S^3$},
  volume = {23},
  ISSN = {0387-3870},
  url = {http://dx.doi.org/10.3836/tjm/1255958682},
  DOI = {10.3836/tjm/1255958682},
  number = {2},
  journal = {Tokyo Journal of Mathematics},
  publisher = {Tokyo Journal of Mathematics},
  author = {Homma,  Yasushi},
  year = {2000},
  month = dec 
}

@article{Gielen2023,
  title = {Operator systems for tolerance relations on finite sets},
  volume = {34},
  ISSN = {0019-3577},
  url = {http://dx.doi.org/10.1016/j.indag.2022.12.009},
  DOI = {10.1016/j.indag.2022.12.009},
  number = {3},
  journal = {Indagationes Mathematicae},
  publisher = {Elsevier BV},
  author = {Gielen,  Mick and van Suijlekom,  Walter D.},
  year = {2023},
  month = may,
  pages = {606–621}
}

@mastersthesis{Koot2021,
  author = {Ian Koot},
  school = {Radboud University Nijmegen},
  title = {Properties of the Propagation number},
  year = {2021}
}

@book{Varshalovich1988,
  title = {Quantum Theory of Angular Momentum},
  ISBN = {9789814415491},
  url = {http://dx.doi.org/10.1142/0270},
  DOI = {10.1142/0270},
  publisher = {WORLD SCIENTIFIC},
  author = {Varshalovich,  Dmitry and Moskalev,  Anatoliy and Khersonskii, Valeriy},
  year = {1988},
  month = oct 
}

@book{SchulzBaldes2023,
  title = {Spectral Flow: A Functional Analytic and Index-Theoretic Approach},
  ISBN = {9783111172477},
  url = {http://dx.doi.org/10.1515/9783111172477},
  DOI = {10.1515/9783111172477},
  publisher = {De Gruyter},
  author = {Doll,  Nora and Schulz-Baldes,  Hermann and Waterstraat,  Nils},
  year = {2023},
  month = may 
}

@book{VanSuijlekom2025,
  title = {Noncommutative Geometry and Particle Physics},
  ISBN = {9783031591204},
  ISSN = {2352-3905},
  url = {http://dx.doi.org/10.1007/978-3-031-59120-4},
  DOI = {10.1007/978-3-031-59120-4},
  journal = {Mathematical Physics Studies},
  publisher = {Springer Nature Switzerland},
  author = {van Suijlekom,  Walter D.},
  year = {2025}
}

@book{Rordam2000,
  title = {An Introduction to K-Theory for C*-Algebras},
  url = {http://dx.doi.org/10.1017/CBO9780511623806.002},
  DOI = {10.1017/cbo9780511623806.002},
  publisher = {Cambridge University Press},
  author = {Rørdam,  Mikael and Larsen,  Flemming and Laustsen,  Niels},
  year = {2000}
}

@BOOK{Khalkhali2009,
  title     = "Basic Noncommutative Geometry",
  author    = "Khalkhali, Masoud",
  publisher = "European Mathematical Society",
  series    = "EMS Series of Lectures in Mathematics",
  month     =  dec,
  year      =  2009,
  address   = "Z{\"u}rich, Switzerland",
  language  = "en"
}

@inbook{CPR2011,
  title = {Spectral triples: examples and index theory},
  ISBN = {9783037195086},
  ISSN = {2943-4939},
  url = {http://dx.doi.org/10.4171/008-1/5},
  DOI = {10.4171/008-1/5},
  booktitle = {Noncommutative Geometry and Physics: Renormalisation,  Motives,  Index Theory},
  publisher = {EMS Press},
  author = {Carey,  Alan L. and Phillips,  John and Rennie,  Adam},
  year = {2011},
  month = jul,
  pages = {175-265}
}

@misc{vdDungen2026,
      title={A K-theoretic note on the spectral localiser}, 
      author={Koen van den Dungen},
      year={2026},
      journal = {arXiv},
      howpublished = {https://arxiv.org/abs/2602.20961},
      eprint={2602.20961},
      archivePrefix={arXiv},
      primaryClass={math.KT},
      url={https://arxiv.org/abs/2602.20961}, 
}

@article{SchulzBaldesSpheres2023,
  title = {The generators of the K-groups of the sphere},
  volume = {41},
  ISSN = {0723-0869},
  url = {http://dx.doi.org/10.1016/j.exmath.2023.125519},
  DOI = {10.1016/j.exmath.2023.125519},
  number = {4},
  journal = {Expositiones Mathematicae},
  publisher = {Elsevier BV},
  author = {Schulz-Baldes,  Hermann and Stoiber,  Tom},
  year = {2023},
  month = dec,
  pages = {125519}
}

@mastersthesis{Hekkelman2021,
  author = {Eva-Maria Hekkelman},
  school = {Radboud University Nijmegen},
  title = {Truncated Geometry on the Circle},
  year = {2021}
}

@article{Cerjan2024,
  title = {Classifying photonic topology using the spectral localizer and numerical K-theory},
  volume = {9},
  ISSN = {2378-0967},
  url = {http://dx.doi.org/10.1063/5.0239018},
  DOI = {10.1063/5.0239018},
  number = {11},
  journal = {APL Photonics},
  publisher = {AIP Publishing},
  author = {Cerjan,  Alexander and Loring,  Terry A.},
  year = {2024},
  month = nov 
}

@article{VanSuijlekom2005,
  title = {The Local Index Formula for SUq(2)},
  volume = {35},
  ISSN = {0920-3036},
  url = {http://dx.doi.org/10.1007/s10977-005-3116-4},
  DOI = {10.1007/s10977-005-3116-4},
  number = {3–4},
  journal = {K-Theory},
  publisher = {Portico},
  author = {van Suijlekom,  Walter D. and Dabrowski,  Ludwik and Landi,  Giovanni and Sitarz,  Andrzej and Várilly,  Joseph C.},
  year = {2005},
  month = aug,
  pages = {375–394}
}

@article{Leimbach24,
title = {Gromov–Hausdorff convergence of spectral truncations for tori},
journal = {Advances in Mathematics},
volume = {439},
pages = {109496},
year = {2024},
issn = {0001-8708},
doi = {https://doi.org/10.1016/j.aim.2024.109496},
url = {https://www.sciencedirect.com/science/article/pii/S0001870824000112},
author = {Malte Leimbach and Walter D. van Suijlekom}
}

@book{GraciaBonda2001,
  title = {Elements of Noncommutative Geometry},
  ISBN = {9781461200055},
  url = {http://dx.doi.org/10.1007/978-1-4612-0005-5},
  DOI = {10.1007/978-1-4612-0005-5},
  publisher = {Birkh\"{a}user Boston},
  author = {Gracia-Bondía,  José M. and Várilly,  Joseph C. and Figueroa,  Héctor},
  year = {2001}
}

@BOOK{Higson2001,
  title     = "Analytic {K-Homology}",
  author    = "Higson, Nigel and Roe, John",
  publisher = "Oxford University Press",
  series    = "Oxford Mathematical Monographs",
  month     =  feb,
  year      =  2001,
  address   = "London, England",
  language  = "en"
}

@misc{VanSuijlekom2024,
  doi = {10.48550/ARXIV.2411.02981},
  url = {https://arxiv.org/abs/2411.02981},
  howpublished = {https://arxiv.org/abs/2411.02981},
  author = {van Suijlekom,  Walter D.},
  title = {Higher K-groups for operator systems},
  publisher = {arXiv},
  year = {2024},
  copyright = {Creative Commons Attribution 4.0 International}
}

@article{Choi1977,
title = {Injectivity and operator spaces},
journal = {Journal of Functional Analysis},
volume = {24},
number = {2},
pages = {156-209},
year = {1977},
issn = {0022-1236},
doi = {https://doi.org/10.1016/0022-1236(77)90052-0},
url = {https://www.sciencedirect.com/science/article/pii/0022123677900520},
author = {Man-Duen Choi and Edward G. Effros}
}

@article{Leimbach2025,
  title = {Convergence of Peter–Weyl truncations of compact quantum groups},
  ISSN = {1661-6960},
  url = {http://dx.doi.org/10.4171/jncg/634},
  DOI = {10.4171/jncg/634},
  journal = {Journal of Noncommutative Geometry},
  publisher = {European Mathematical Society - EMS - Publishing House GmbH},
  author = {Leimbach,  Malte},
  year = {2025},
  month = nov 
}

@article{GaudillotEstrada2025,
  title = {Convergence of Spectral Truncations for Compact Metric Groups},
  volume = {2025},
  ISSN = {1687-0247},
  url = {http://dx.doi.org/10.1093/imrn/rnaf197},
  DOI = {10.1093/imrn/rnaf197},
  number = {13},
  journal = {International Mathematics Research Notices},
  publisher = {Oxford University Press (OUP)},
  author = {Gaudillot--Estrada,  Yvann and van Suijlekom,  Walter D.},
  year = {2025},
  month = jul 
}

@misc{Peng2026,
  doi = {10.48550/ARXIV.2604.27475},
  url = {https://arxiv.org/abs/2604.27475},
  howpublished = {https://arxiv.org/abs/2604.27475},
  author = {Peng,  Xintao and Wang,  Qin},
  title = {Gromov-Hausdorff Convergence of Spectral Truncations for Quantum Groups},
  publisher = {arXiv},
  year = {2026},
  copyright = {Creative Commons Attribution 4.0 International}
}

@article{Hamana1979,
  title={Injective envelopes of operator systems},
  author={Hamana, Masamichi},
  journal={Publications of the Research Institute for Mathematical Sciences},
  volume={15},
  number={3},
  pages={773--785},
  year={1979}
}

@book{Prodan2016,
  title = {Bulk and Boundary Invariants for Complex Topological Insulators: From K-Theory to Physics},
  ISBN = {9783319293516},
  ISSN = {2352-3905},
  url = {http://dx.doi.org/10.1007/978-3-319-29351-6},
  DOI = {10.1007/978-3-319-29351-6},
  journal = {Mathematical Physics Studies},
  publisher = {Springer International Publishing},
  author = {Prodan,  Emil and Schulz-Baldes,  Hermann},
  year = {2016}
}

@article{Loring2020,
  title = {The spectral localizer for even index pairings},
  volume = {14},
  ISSN = {1661-6960},
  url = {http://dx.doi.org/10.4171/JNCG/357},
  DOI = {10.4171/jncg/357},
  number = {1},
  journal = {Journal of Noncommutative Geometry},
  publisher = {European Mathematical Society - EMS - Publishing House GmbH},
  author = {Loring,  Terry A. and Schulz-Baldes,  Hermann},
  year = {2020},
  month = Feb,
  pages = {1–23}
}

@misc{Loring2017,
  doi = {10.48550/ARXIV.1701.07455},
  url = {https://arxiv.org/abs/1701.07455},
  author = {Loring,  Terry and Schulz-Baldes,  Hermann},
  title = {Finite volume calculation of $K$-theory invariants},
  publisher = {arXiv},
  year = {2017},
  copyright = {arXiv.org perpetual,  non-exclusive license},
  howpublished={https://arxiv.org/abs/1701.07455}
}

@article{Loring2019,
  title = {Spectral Flow Argument Localizing an Odd Index Pairing},
  volume = {62},
  ISSN = {1496-4287},
  url = {http://dx.doi.org/10.4153/CMB-2018-013-x},
  DOI = {10.4153/cmb-2018-013-x},
  number = {02},
  journal = {Canadian Mathematical Bulletin},
  publisher = {Canadian Mathematical Society},
  author = {Loring,  Terry A. and Schulz-Baldes,  Hermann},
  year = {2019},
  month = Jan,
  pages = {373–381}
}

@article{Eleftherakis2026,
  title = {Morita equivalence for operator systems},
  ISSN = {1565-8538},
  url = {http://dx.doi.org/10.1007/s11854-025-0428-6},
  DOI = {10.1007/s11854-025-0428-6},
  journal = {Journal d’Analyse Mathématique},
  publisher = {Springer Science and Business Media LLC},
  author = {Eleftherakis,  George K. and Kakariadis,  Evgenios T. A. and Todorov,  Ivan G.},
  year = {2026},
  month = July 
}

@BOOK{Edmonds1996-xa,
  title     = "Angular momentum in quantum mechanics",
  author    = "Edmonds, A R",
  publisher = "Princeton University Press",
  series    = "Princeton Landmarks in Mathematics and Physics",
  month     =  jan,
  year      =  1996,
  address   = "Princeton, NJ",
  language  = "en"
}

@article{DAndrea2014,
  title = {Spectral geometry with a cut-off: Topological and metric aspects},
  volume = {82},
  ISSN = {0393-0440},
  url = {http://dx.doi.org/10.1016/j.geomphys.2014.03.014},
  DOI = {10.1016/j.geomphys.2014.03.014},
  journal = {Journal of Geometry and Physics},
  publisher = {Elsevier BV},
  author = {D’Andrea,  Francesco and Lizzi,  Fedele and Martinetti,  Pierre},
  year = {2014},
  month = Aug,
  pages = {18–45}
}

@article{vanSuijlekom2021,
  title = {Gromov–Hausdorff convergence of state spaces for spectral truncations},
  volume = {162},
  ISSN = {0393-0440},
  url = {http://dx.doi.org/10.1016/j.geomphys.2020.104075},
  DOI = {10.1016/j.geomphys.2020.104075},
  journal = {Journal of Geometry and Physics},
  publisher = {Elsevier BV},
  author = {van Suijlekom,  Walter D.},
  year = {2021},
  month = Apr,
  pages = {104075}
}

@article{Rieffel2023,
  title = {Convergence of Fourier truncations for compact quantum groups and finitely generated groups},
  volume = {192},
  ISSN = {0393-0440},
  url = {http://dx.doi.org/10.1016/j.geomphys.2023.104921},
  DOI = {10.1016/j.geomphys.2023.104921},
  journal = {Journal of Geometry and Physics},
  publisher = {Elsevier BV},
  author = {Rieffel,  Marc A.},
  year = {2023},
  month = Oct,
  pages = {104921}
}

@article{Li2026,
  title = {The odd spectral localiser via asymptotic morphisms and quasi-projections},
  volume = {11},
  ISSN = {2379-1683},
  url = {http://dx.doi.org/10.2140/akt.2026.11.213},
  DOI = {10.2140/akt.2026.11.213},
  number = {2},
  journal = {Annals of K-Theory},
  publisher = {Mathematical Sciences Publishers},
  author = {Li,  Yuezhao and Mesland,  Bram},
  year = {2026},
  month = Mar,
  pages = {213–260}
}

@article{Connes2022,
  title = {Tolerance relations and operator systems},
  volume = {88},
  ISSN = {2064-8316},
  url = {http://dx.doi.org/10.1007/s44146-022-00012-3},
  DOI = {10.1007/s44146-022-00012-3},
  number = {1-2},
  journal = {Acta Scientiarum Mathematicarum},
  publisher = {Springer Science and Business Media LLC},
  author = {Connes,  Alain and van Suijlekom,  Walter D.},
  year = {2022},
  month = Aug,
  pages = {101–129}
}

@misc{VanSuijlekom2024b,
  doi = {10.48550/ARXIV.2409.02773},
  url = {https://arxiv.org/abs/2409.02773},
  howpublished = {https://arxiv.org/abs/2409.02773},
  author = {van Suijlekom,  Walter D.},
  title = {A generalization of K-theory to operator systems},
  publisher = {arXiv},
  year = {2024},
  copyright = {Creative Commons Attribution 4.0 International}
}

@article{Farenick2021,
  title = {The operator system of Toeplitz matrices},
  volume = {8},
  ISSN = {2330-0000},
  url = {http://dx.doi.org/10.1090/btran/83},
  DOI = {10.1090/btran/83},
  number = {32},
  journal = {Transactions of the American Mathematical Society,  Series B},
  publisher = {American Mathematical Society (AMS)},
  author = {Farenick,  Douglas},
  year = {2021},
  month = Nov,
  pages = {999–1023}
}

@misc{CerjanSB2026,
  doi = {10.48550/ARXIV.2506.14174},
  url = {https://arxiv.org/abs/2506.14174},
  howpublished = {https://arxiv.org/abs/2506.14174},
  author = {Cerjan,  Alexander and Schulz-Baldes,  Hermann},
  title = {Detecting local topology with the spectral localizer},
  publisher = {arXiv},
  year = {2025},
  copyright = {arXiv.org perpetual,  non-exclusive license}
}

@book{Efthimiou2014,
  title = {Spherical Harmonics in p Dimensions},
  ISBN = {9789814596701},
  url = {http://dx.doi.org/10.1142/9134},
  DOI = {10.1142/9134},
  publisher = {WORLD SCIENTIFIC},
  author = {Efthimiou,  Costas and Frye,  Christopher},
  year = {2014},
  month = Feb 
}

@article{CHARRON2025110792,
title = {On the concentration of the Fourier coefficients for products of Laplace-Beltrami eigenfunctions on real-analytic manifolds},
journal = {Journal of Functional Analysis},
volume = {288},
number = {5},
pages = {110792},
year = {2025},
issn = {0022-1236},
doi = {https://doi.org/10.1016/j.jfa.2024.110792},
url = {https://www.sciencedirect.com/science/article/pii/S0022123624004804},
author = {Philippe Charron and François Pagano},
}

@article{Morrey1957,
  title = {On the analyticity of the solutions of linear elliptic systems of partial differential equations},
  volume = {10},
  ISSN = {1097-0312},
  url = {http://dx.doi.org/10.1002/cpa.3160100204},
  DOI = {10.1002/cpa.3160100204},
  number = {2},
  journal = {Communications on Pure and Applied Mathematics},
  publisher = {Wiley},
  author = {Morrey,  C. B. and Nirenberg,  L.},
  year = {1957},
  month = Jan,
  pages = {271–290}
}

@book{Lewis2023,
  title = {Geometric Analysis on Real Analytic Manifolds},
  ISBN = {9783031379130},
  ISSN = {1617-9692},
  url = {http://dx.doi.org/10.1007/978-3-031-37913-0},
  DOI = {10.1007/978-3-031-37913-0},
  journal = {Lecture Notes in Mathematics},
  publisher = {Springer International Publishing},
  author = {Lewis,  Andrew D.},
  year = {2023}
}

@misc{Bhatt2026,
  doi = {10.48550/ARXIV.2610.01132},
  url = {https://arxiv.org/abs/2610.01132},
  howpublished = {https://arxiv.org/abs/2610.01132},
  author = {Bhatt,  Vibhor and Guin,  Satyajit and Sharma,  Palash},
  title = {Gromov-Hausdorff Convergence of Spectral Truncations for Noncommutative Tori},
  publisher = {arXiv},
  year = {2026},
  copyright = {arXiv.org perpetual,  non-exclusive license}
}

\end{document}